\documentclass{article}

\usepackage{arxiv}
\usepackage{float}
\usepackage{amsmath,amssymb,amsthm,amsfonts,dsfont}
\usepackage{natbib}
\setcitestyle{numbers,square}
\newtheorem{theorem}{Theorem}[section]
\newtheorem{corollary}{Corollary}[section]
\newtheorem{lemma}[theorem]{Lemma}
\newtheorem{proposition}[theorem]{Proposition}
\newtheorem{remark}[theorem]{Remark}
\usepackage{mathtools}
\usepackage{bbm}
\usepackage{listings}
\usepackage{esint}
\usepackage{soul}
\usepackage[table,xcdraw]{xcolor}
\usepackage{pifont}
\usepackage{float}
\usepackage{url}
\usepackage{varwidth}
\usepackage{dutchcal}
\usepackage{mathrsfs}
\usepackage{caption}
\usepackage{subcaption}
\usepackage{multirow}
\usepackage{longtable}
\usepackage[utf8]{inputenc} 
\usepackage[T1]{fontenc}    
\usepackage{hyperref}       
\usepackage{url}            
\usepackage{booktabs}       
\usepackage{amsfonts}       
\usepackage{nicefrac}       
\usepackage{microtype}      
\usepackage{lipsum}		
\usepackage{graphicx}
\usepackage{natbib}
\usepackage{doi}

\def\given{\typeout{Command 'given' should only be used within bracket command}}
\newcounter{@bracketlevel}
\def\@bracketfactory#1#2#3#4#5#6{
\expandafter\def\csname#1\endcsname##1{%
\addtocounter{@bracketlevel}{1}%
\global\expandafter\let\csname @middummy\alph{@bracketlevel}\endcsname\given%
\global\def\given{\mskip#5\csname#4\endcsname\vert\mskip#6}\csname#4l\endcsname#2##1\csname#4r\endcsname#3%
\global\expandafter\let\expandafter\given\csname @middummy\alph{@bracketlevel}\endcsname
\addtocounter{@bracketlevel}{-1}}%
}
\def\bracketfactory#1#2#3{%
\@bracketfactory{#1}{#2}{#3}{relax}{1mu plus 0.25mu minus 0.25mu}{0.6mu plus 0.15mu minus 0.15mu}
\@bracketfactory{b#1}{#2}{#3}{big}{1mu plus 0.25mu minus 0.25mu}{0.6mu plus 0.15mu minus 0.15mu}
\@bracketfactory{bb#1}{#2}{#3}{Big}{2.4mu plus 0.8mu minus 0.8mu}{1.8mu plus 0.6mu minus 0.6mu}
\@bracketfactory{bbb#1}{#2}{#3}{bigg}{\ref{th:LBk}mu plus 1mu minus 1mu}{2.4mu plus 0.75mu minus 0.75mu}
\@bracketfactory{bbbb#1}{#2}{#3}{Bigg}{4mu plus 1mu minus 1mu}{3mu plus 0.75mu minus 0.75mu}
}
\bracketfactory{clc}{\lbrace}{\rbrace}
\bracketfactory{clr}{(}{)}
\bracketfactory{cls}{[}{]}
\bracketfactory{abs}{\lvert}{\rvert}
\bracketfactory{norm}{\Vert}{\Vert}
\bracketfactory{floor}{\lfloor}{\rfloor}
\bracketfactory{ceil}{\lceil}{\rceil}
\bracketfactory{angle}{\langle}{\rangle}

\usepackage{xpatch}
\xpatchbibmacro{cite}{\printnames{labelname}}%
{\ifciteseen{\printnames{labelname}}{\printnames[][-\value{listtotal}]{labelname}}}
{}{}

\xpatchbibmacro{textcite}{\printnames{labelname}}%
{\ifciteseen{\printnames{labelname}}{\printnames[][-\value{listtotal}]{labelname}}}
{}{}

\newcommand{\EE}{\mathbbm{E}}

\newcommand{\wt}{\widetilde}

\newcommand{\var}{\mathrm{Var}}

\newcommand{\cN}{\mathcal{N}}

\newcommand{\cF}{\mathcal{F}}

\newcommand{\IP}{\mathbb{P}}

\newcommand{\E}{\mathbbm{E}}

\newcommand{\al}{\alpha}

\newcommand{\eps}{\varepsilon}

\newcommand\numberthis{\addtocounter{equation}{1}\tag{\theequation}}
\makeatletter
\newcommand*{\rom}[1]{\expandafter\@slowromancap\romannumeral #1@}
\makeatother
\newcommand{\PP}
{\mathbbm{P}}

\title{Isolated vertices in two duplication-divergence models with edge deletion}

\author{Tiffany Y. Y. Lo\\
	Department of Mathematics\\
	Stockholm University\\
	lbanovägen 28, SE-106 91 Stockholm, Sweden \\
	\texttt{tiffany.y.y.lo@math.su.se} \\
	\And
	{Gesine Reinert} \\
	Department of Statistics\\
	University of Oxford\\
	4-29 St Giles, Oxford OX1 3LB, UK \\
	\texttt{reinert@stats.ac.uk} \\
	\And
	{Ruihua Zhang} \\
	Department of Statistics\\
	University of Oxford\\
	4-29 St Giles, Oxford OX1 3LB, UK \\
	\texttt{ruihua.zhang@stats.ac.uk} \\
}

\renewcommand{\headeright}{}
\renewcommand{\undertitle}{}

\begin{document}
\maketitle
\title{Isolated vertices in two duplication-divergence models with edge deletion}

\begin{abstract}
Duplication-divergence models are a
popular model for the evolution of gene and protein interaction networks.  However, existing duplication-divergence models often neglect realistic features such as loss of interactions. Thus, in this paper we present two novel models that incorporate  random edge deletions into the duplication-divergence framework. 
As in protein-protein interaction networks, with proteins as vertices and interactions as edges, by design isolated vertices 
tend to be rare, our main focus is on the number of isolated vertices; our main result gives lower and upper bounds for the proportion of isolated vertices, when the network size is large. Using these bounds we identify the parameter regimes for which almost all vertices are typically isolated; and also show that there are parameter regimes in which the proportion of isolated vertices can be bounded away from 0 and 1 with high probability. In addition, we find regimes in which the proportion of isolated vertices tends to be small. The proof relies on a standard martingale argument, which in turn requires a careful analysis of the first two moments of the expected degree distribution. 
The theoretical findings are illustrated by simulations, indicating that as the network size tends to infinity, the proportion of isolated vertices can converge to a limit that is neither 0 or 1.
\end{abstract}


\section{Introduction}\label{sec:back}

Networks are fundamental structures for describing complex systems. One example is given by  physical interactions between proteins in an organism, which can be represented by a protein-protein interaction (PPI) networks with proteins as vertices and edges representing physical interactions. In the light of evolution of species, a  natural question is how to model the evolution of such networks. For this purpose, duplication-divergence (DD) models have been explored
\citep{ChungLu2003, Sole2022, Gibson2011}.

The standard duplication-divergence model \cite{ChungLu2003, Bhan2002} 
runs as follows. Starting from the initial graph $G_{m_0}$ on $m_0$ vertices, 
the graph $G_m$, $m\ge m_0$, is generated by first choosing a vertex $v$ uniformly at random in $G_{m-1}$. 
A new vertex $w$ is then added to $G_{m-1}$ and connected to all neighbours of $v$ ({\it a duplication step}). Then, the edges between $w$ and the neighbours of $v$ 
are independently removed, each with probability $p$ ({\it a divergence step}).
\cite{hermann2016} first observed that a birth-catastrophe process is embedded in the evolution of the  expected degree distributions over time $m$, and this connection is subsequently also exploited in \cite{BarbourL2022, jordan2018}.
For $p \ge p^* \approx 0.43286$, 
so that $1-p^*$ is the unique solution of $x e^x=1$, \cite{hermann2016} showed that nearly all vertices become isolated as $m \to \infty$.
When $p> 1-e^{-1}$, \cite{jordan2018} proves that the degree distribution of the connected component has a limit that can be characterised by the stationary distribution of a birth-catastrophe process, and this is  used in \cite{jacquet2020power} to prove a power-law behaviour for the degree distribution.  The $p<p^*$ case is also studied in detail in \cite{BarbourL2022}, proving a central limit theorem for the logarithm of the degree of a  uniformly chosen vertex.  Other features of the model, such as the count of cliques and the degree of a fixed vertex,  are  studied in \cite{frieze2020degree,hermann2016}. 
Modifications to the model include that new vertices may connect to vertices which are not involved in the duplication-divergence step \cite{pastor2003evolving, bebek2006degree}, or they may connect directly to the duplicated vertex \cite{lambiotte2016structural}. For these variants, \cite{BarbourL2022, lls2024} show that if $p$ is large enough, the expected degree distribution converges to a non-negative sequence $(a_i)_{i \geq 0}$ with $\sum_{i\ge 0} a_i=1$ and $a_0<1$,  as $m \to \infty$.

Evolving PPI networks may exhibit not only divergence but also edge loss, for example by protein interactions losing functionality \cite{David2012, Yue2005}. Random edge deletions can significantly affect network behaviour; for instance, \cite{deijfen2009growing} shows that such a mechanism can lead to different asymptotic degree distributions in preferential attachment models. DD models with edge deletions are also considered in \cite{BarbourL2022,hermann2021}, but in a continuous-time setting. Both \cite{BarbourL2022,hermann2021} show that the expected degree distribution of these models are qualitatively similar to that of the standard models. Another model is introduced in \cite{Thornblad2014}, where starting from a single isolated vertex, a vertex $v$ is chosen at random at each step, and a coin toss then determines whether a new vertex $u$ is added and connected to $v$ and its neighbours, or all edges of $v$ (if any) are deleted. Almost sure convergence of the degree distribution is proved in \cite{Thornblad2014} by leveraging the connection between $k$-cliques and the number of vertices of degree $k-1$. However, the clique assumption limits the usefulness of the model for understanding real-world networks. 

Here we generalise the duplication-divergence model in \cite{lambiotte2016structural} by incorporating random edge deletions, hence extending both the
model in \cite{lambiotte2016structural} and that in \cite{Thornblad2014}. We mainly focus on the number of isolated vertices, for the following reason. In organisms, proteins interact in order to achieve biological functionalities. Hence proteins which do not have any interactions with other proteins tend to be rare, and thus the number of isolated vertices can be used to assess DD models, see for example \cite{zhang2024simulating}. 

The paper is structured as follows. Section \ref{sec:newmodel} introduces our two new DD models with edge deletion. Section \ref{sec:results} collates the results and illustrates them with simulations. Our proofs rely heavily on recursion formulas which are derived in Section \ref{sec:recursion}. A second main ingredient to our proofs is a martingale lemma which is given and proved in Section \ref{sec:backhausz}. The key results are then proved in Section \ref{section: convergence0}, and Section \ref{sec:furtherproofs} contains further proofs. Additional simulation results are available in \textbf{the Appendix.}

\section{Two Duplication-Divergence Models with Edge Deletion}
\label{sec:newmodel}

We propose two DD models with edge deletion, each having three parameters $0\le p,q,r \le 1$. For a positive integer $n$ denote $[n]:=\{1,\dots, n\}$.
The main model, {\bf Model A}, is as follows. At time $m_0$,
we start with a simple graph $G_{m_0}$ on $m_0 \ge  1$ vertices with labels $[m_0]$. At integer times $m \ge   m_0$, the graph $G_m$, with vertex set $V(G_m)=[m]$, is generated from $G_{m-1}$ as follows: 

\begin{itemize}
    \item With probability $q$, we perform a \textit{duplication-and-divergence step}: choose a vertex $v$ uniformly at random from $[m-1]$, and add a new vertex labelled $m$ to $G_{m-1}$. Draw an edge between $m$ and $v$, and between $m$ and each neighbour of $v$ (if any).  Independently, erase each edge connecting vertex $m$ and a neighbour of~$v$ with probability $p$, and 
    the edge connecting $m$ and $v$ with probability $1-r$. 
    \item With probability $1-q$, we perform a \textit{deletion-and-addition step}: choose uniformly at random a vertex from $[m-1]$, delete all the edges connected to it, but keep the vertex and its former neighbours. Further, we add an isolated vertex.
\end{itemize}

The motivation for adding an isolated vertex at the deletion-and-addition step is to keep the total number of vertices  fixed at $m$ after $m-m_0$ steps, 
mainly for mathematical tractability. We also consider a variant of Model A, where the number $N_m$ of vertices after $m{-m_0}$ time steps is
random. In this model, {\bf Model B}, the initial network $G_{m_0}$ again consists of $m_0 \ge  1$ vertices with labels $[m_0]$. At integer times $m \ge   m_0$, the graph $G_m$, with vertex set $V(G_m)$, is generated from $G_{m-1}$ as follows:
\begin{itemize}
    \item With probability $q_{m-1}$, we perform a \textit{duplication-divergence step} as in Model A, but with $p$ and $r$ replaced with $p_{m-1}$ and $r_{m-1}$. 
    \item With probability $1-q_{m-1}$, we perform a \textit{deletion step}: erase all edges (if any) of a uniformly chosen vertex, but keep the vertex and its neighbours in the network. 
\end{itemize}
  
In Model B, we do not add any new vertices to the network in a deletion step. Hence $N_m - m_0$ follows a Poisson-binomial distribution, as at each step $m \geq m_0$, a new vertex is added with probability $q_m$, independently of other events. 
In either model, the degree of any vertex in $G_m$ is at most $m-1$.

\begin{remark}
Models A and B are related to models that have been studied elsewhere.
\begin{enumerate}
    \item When $q=1$ and $r=0$, Model A 
    is the standard DD model in \cite{Bhan2002,ChungLu2003}.
    \item When $q=r=1$, Model A  
    is the DD model considered in \cite{lambiotte2016structural,lls2024}.
    \item When $p=q=1$ and $G_{m_0}$ is a single vertex,  Model A 
    can be seen as a Bernoulli bond percolation on random recursive tree, where each edge is removed independently with probability $1-r$; see e.g.\ \cite{bertoin2014sizes}, where the cluster sizes are investigated.  \item  When $p_m=0$, $r_m=1$, and $q_m=q$ for all $m \ge m_0$, and if $G_{m_0} $ consists of isolated cliques, then Model B is the model in \cite{Thornblad2014}.
\end{enumerate}  
\end{remark}

\section{Results}\label{sec:results}

For both models, let $\mathcal{N}_{m,k}$ denote the set of vertices of degree $k$ in $G_m$, with $N_{m,k} = |\mathcal{N}_{m,k}|$ as the count of such vertices, and $N_{m,k}=0$ for $k\ge m$. 
Let $N_m:=\sum^{m-1}_{k=0}N_{m,k}$  be the total number of vertices in $G_m$, and unless specified, $C>0$ is a constant not depending $m$, but it may depend on $p,q,r,m_0$ and $(N_{m_0,j})_{0\le j\le m_0-1}$, and it may vary from one 
instance to another.

If $q = q_m = 0$ for $m \geq m_0$, no new edges are created, and edges of non-isolated vertices are deleted when chosen. In Model A, an isolated vertex is added at each step and remains isolated, ensuring at least $m - m_0$ isolated vertices at time $m$. In Model B, the vertex count is fixed at $m_0$, and thus all vertices are chosen within a finite time almost surely. Consequently, in both
models, $N_{m,0}/N_m \to 0$ almost surely as $m \to \infty$ if $q  = 0$. 
Our results below mainly concern Model A, and include the $q=0$ case. We begin with a bound on the expected number of isolated vertices. The proof, which is standard, is deferred to Section \ref{section: convergence0}.

\begin{theorem} \label{lem:meanbounds}
For Model $A$, assume that $q$ and $r$ are such that 
$ u := 1
-2q(1-r) > 0.$ Then, regardless of the initial graph $G_{m_0}$,
\begin{align*}
    \frac{2(1-q)}{u} - O(m^{-u})\le \EE \left( \frac{N_{m,0}}{m} \right) \le  \frac{3(1-q) + pq(1-r) }{u} + O(m^{-u}). 
\end{align*}
\end{theorem}

We also establish lower and upper bounds on $N_{m,0}/m$ in Model A, whose proof in Section \ref{section: convergence0} is an adaptation of the martingale argument in \cite{Backhausz, Thornblad2014} that is used to show concentration of degree proportions. To state the theorems we introduce 
    \begin{align} \tau&:=\tau(p,q)=6q-4pq-3+p^2q; \label{tau}\\
       \kappa &:=\kappa(p,q)= 4 q - 2 p q - 2;  \label{kappa}\\
        \label{de:rho}
         \rho_0&:=\rho_0(q,r)=(1-q)/\{ 1-q(1-r)\};\\ 
          \rho_1&:= \rho_1(p,q,r)= \{ 3 (1-q) + pq(1-r) \} / \{ 2 [1 - q(1-r)]\};  
          \label{rho1} \\
\theta_0&:=\theta_0(q,r)=\{ 2(1-q)+q(1-r)\}/\{ 2-q(1-r)\}; \label{theta}\\
\theta_1&:=\theta_1(q,r)=\{3(1-q)+q(1-r)\} /\{2-q(1-r)\} \label{theta1}
          .
    \end{align}

\begin{theorem}
\label{th:LBk} In Model A, suppose that $m_0\ge 2$ and the parameters $p,q,r$ are such that $q \leq 
   \min \{1, 1/(2(1-r))\}$ and $\tau<1$. Then, regardless of the initial graph $G_{m_0}$, there exist constants $C>0$ and $0<\delta<\min\{1,1-\max\{\tau,\kappa\}\}$ such that {with probability at least $1-Cm^{-\delta/2}$,}
     \begin{align}\label{eq:propLB}
     \rho_0-Cm^{-\delta/4} \le \frac{N_{m,{0}}}{m} \le \rho_1 + {Cm^{-\delta/4}}.
     \end{align}
\end{theorem} 

The condition $m_0\ge 2$ streamlines the proof but is inconsequential since the result is valid for any $G_{m_0}$. For certain $(p,q,r)$ , $\rho_0=0$ or $\rho_1\ge 1$; 
we include these trivial cases in the theorem for completeness. As examples, for $r\ge 1/2$, $\rho_0=0$ in the case $q=1$; and for $r=1$,
$\rho_1\ge 1$  if $q\le 1/3$. When $r=1$, $\rho_1-\rho_0=(1-q)/2$, thus giving a good estimate of $N_{m,0}/m$ when $m,q$ are large enough; see Figure \ref{fig:modelA}. It is also possible to improve the bounds above for certain $(p,q,r)$, but we do not attempt this here.

\begin{remark} We collect the following observations for the theorem above. 
    \begin{enumerate}
    \item Figure \ref{fig:pqr}  
    illustrates the region for which $\tau < 1$, and also the region for which $q \le 1/(2(1-r))$ when $r <1$. The figure shows that the two inequalities do not trivially imply each other; they also illustrate the large range for which Theorem \ref{th:LBk} can be applied. 
   \begin{figure}[!tpb]
   \centering
    \includegraphics[width=\textwidth, height=0.33\textwidth]{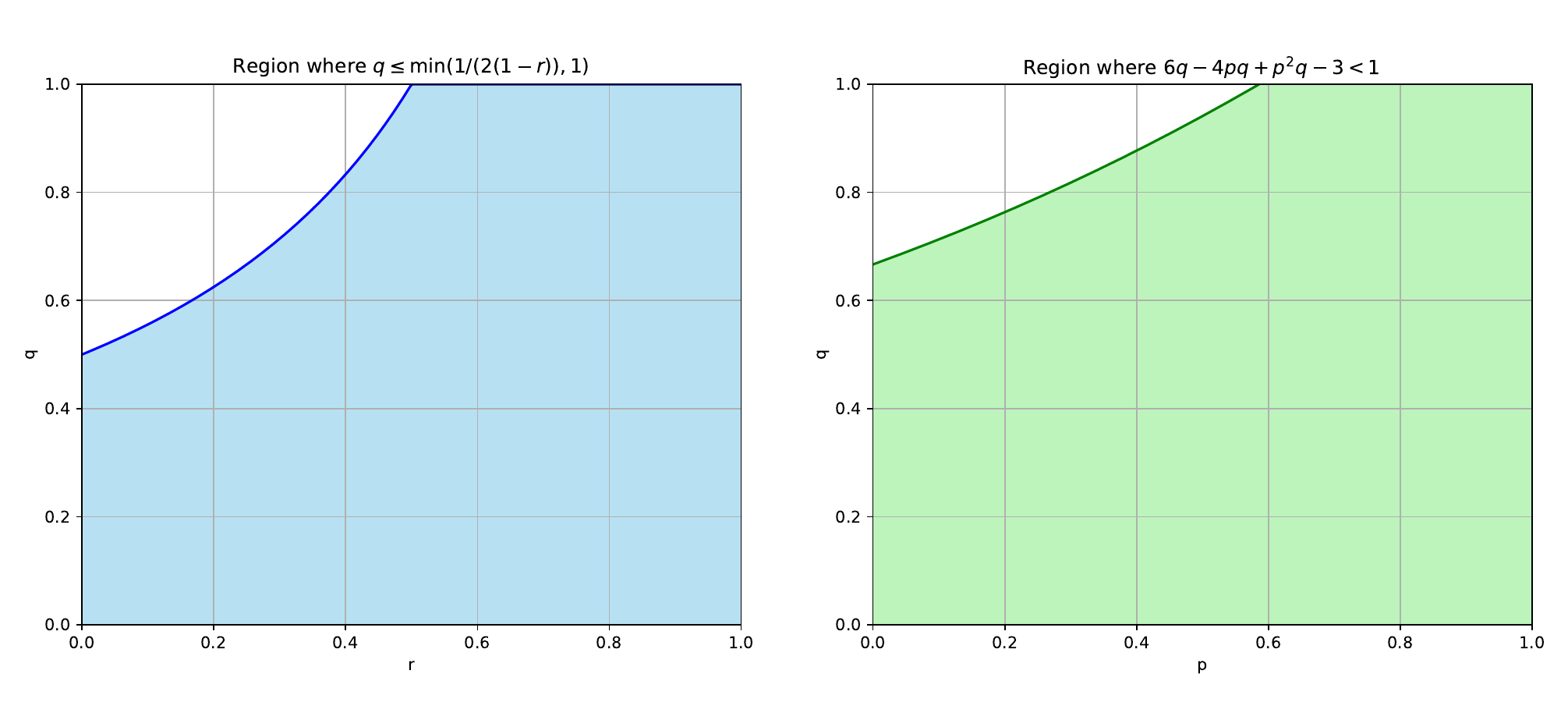}
    \caption{\small{The regions for which the conditions of Theorem \ref{th:LBk} are satisfied. The blue region on the left shows the values of $q,r$ such that $q\le 1/(2(1-r))$, and the green region on the right shows values of $p,q$ such that $\tau<1$.}}
    \label{fig:pqr}
    \vspace{-5mm}
\end{figure}

\item If $r=0$, then  Theorem \ref{th:LBk}  requires $q\le 1/2$, and so the theorem does not cover the standard DD model (with $r=0$ and $q=1$) considered in \cite{BarbourL2022,hermann2016, jordan2018}.

 \item  If $q=r=1$, then Model A is the DD model studied in \cite{lambiotte2016structural,lls2024}. In this case, $N_{m,0}/m\to 0$ a.s.\ as $m\to\infty$, regardless of the initial graph. This follows from the fact that every newly created vertex has at least one edge after the duplication-divergence step, and these edges will not be erased in the subsequent steps.
 
 \item 
 The bound \eqref{eq:propLB} does not hold for all choices of $p,q$ and $r$. For example, if $q=1$ and $p=r=0$, a choice which does not satisfy the assumptions of Theorem \ref{th:LBk},
 the limiting behaviour depends on the initial graph $G_{m_0}$. 
 To see this, if $q=1$ and $p=0$ then every step is a duplication step, all edges to neighbouring vertices are kept, and no edges are ever deleted.
So if $r=0$, and $G_{m_0}$ is a collection of isolated vertices, the graph $G_m$ will consist of $m$ isolated vertices. In contrast, if $G_{m_0}$ is a complete graph, then every vertex will have degree at least $m_0-1$. 
\end{enumerate}
\end{remark}

We can also find $(p,q,r)$ for which the network typically consists of almost all isolated vertices, when the network size $m$ is large, and $(p,q,r)$ for which the proportion of non-isolated vertices is typically non-zero as $m\to\infty$. The next corollary follows immediately from Theorem \ref{th:LBk}.

\begin{corollary}\label{cor}
    For Model A, suppose that $m_0\ge 2$ and that $(p,q,r)$ satisfy the assumptions in Theorem \ref{th:LBk}. The following statements hold regardless of the initial network $G_{m_0}$.
   \begin{enumerate}
       \item  Assume  further  
    $r=0$ and $q \le 1/2$. Then $\rho_0=1$ in \eqref{de:rho}, and
   there exist constants $C>0$ and $0<\delta<\min\{1,1-\max\{\tau,\kappa\}\}$ such that  $ N_{m,0}/m \ge   1-Cm^{-\delta/4}$ with probability at least $1-Cm^{-\delta/2}$.
   \item \label{CU}
   Assume further
    $q>1/(3-(p+2)(1-r))$. Then $\rho_1<1$ in  \eqref{rho1}, and 
    there exist constants $C>0$ and $0<\delta<\min\{1,1-\max\{\tau,\kappa\}\}$ such that $ N_{m,0}/m < \rho_1 + Cm^{-\delta/4}$ with probability at least $1-Cm^{-\delta/2}$.
      \end{enumerate}
\end{corollary}

 Item 1 in Corollary \ref{cor}
 holds regardless of the choice of the edge deletion probability $p$. This is in contrast to the standard DD model (with $q=1$ and $r=0$) considered in \cite{BarbourL2022,hermann2016, jordan2018}, where almost all vertices are isolated only if $p\in[p^*,1]$, where $p^*\approx 0.43286$. An example of $(p,q,r)$ that satisfies the assumption of Item 2 in Corollary \ref{cor} 
 is $r=1$, $q\ge 1/3$, and then $p$ can be any value in $[0,1]$.

When $p=1$,  bounds tighter than $\rho_0$ and $\rho_1$ in Theorem \ref{th:LBk} can be achieved. In Theorem \ref{th:LBK1}, note that $\theta_0=1$ if $q=0$ or (and) $r=0$. Since $p=1$, in these cases, in fact $N_{m,0}/m\to 1$ almost surely as $m\to\infty$. The proof is similar to that of Theorem \ref{th:LBk} and is deferred to Section \ref{section: convergence0}.

\begin{theorem}[The $p=1$ case]\label{th:LBK1}
    For Model A, suppose that $m_0\ge 2$ and $p=1$.
    Then, regardless of the initial network $G_{m_0}$, for any $0<\delta<1$, there exists a constant $C>0$  such that with probability at least $1-Cm^{-\delta/2}$,
     \begin{align}\label{eq:propLB1}
        \theta_0 - Cm^{-\delta/4}\le  \frac{N_{m,0}}{m} \le \theta_1 + Cm^{-\delta/4}.
     \end{align}
\end{theorem}

Note that if $p=0$ and $q=r=1$, then any newly added vertex has degree at least one, so $\lim_{m\to\infty}N_{m,0}/N_m=\lim_{m\to\infty}N_{m,0}/m=0$ almost surely. The next proposition covers this case and more general $r$. The proof, which is standard, can be found in Section \ref{sec:furtherproofs}.

\begin{proposition} [The $p=0, q=1$ case]\label{prop:modelB}
For Model A with $p=0$, $q=1$, and $r\le1$, $m^{-(1-2r)}N_{m,0}$ converges almost surely to a finite limit as $m \to \infty$. Additionally, if $r \ge 1/2$, then $N_{m,0}$ is bounded in expectation as $m\to\infty$; and if $r  >0$, then $N_{m,0}/m$
 converges to 0 in probability.
\end{proposition}

While the above results are bounds, the next result fully addresses the existence of $\lim_{m \to \infty} N_{m,k}/m$ for $k \geq 0$ when $p = q = 1$. 
In this case, the new vertex undergoes only the duplication step.
If the initial network is a single isolated vertex and $r = 1$, the model is the random recursive tree, where it is well-known that $\lim_{m \to \infty} N_{m,j}/m = 2^{-j}$ almost surely for all $j \geq 1$; see \cite{bollobas2001degree}. Instead of using the argument from \cite{bollobas2001degree}, we follow that of \cite{Backhausz, Thornblad2014} to extend this result to $r < 1$, providing a new proof that highlights the challenges when $p, q < 1$. The proof is found in Section \ref{sec:furtherproofs}.

\begin{theorem}[The $p=0, q=1$ case]\label{thm: pq1case} 
    For Model A with  
    $p=q=1$ and $0\le r\le 1$, we have
    for all $k\geq 0$, 
\begin{equation}\label{eq:pq1lim}
        \lim_{m\to\infty} \frac{N_{m,k}}{m} = a_k\quad \text{almost surely,}
    \end{equation}
    where 
    \begin{align}\label{ak}
        a_0 = \frac{1-r}{1+r}\quad\text{and}\quad a_k = \bbclr{\frac{r}{1+r}}^k(1+a_0), \quad k\ge 1, \quad \text{with}\quad \sum_{k\geq 0}a_k=1.
    \end{align}
\end{theorem}
 When $r=1$, then $(a_k)_{k\geq 1}$ in \eqref{ak} is the probability mass function of the geometric distribution on $\{1,2,\dots\}$ with success probability $1/2$.
 
Next, we turn to Model B, for which we show that there are $(p_k,q_k,r_k)_{k \ge  m_0}$ such that the number of isolated vertices remains bounded in expectation as the network size tends to infinity. In the first proposition below, the probabilities $p_k$, $q_k$ and $r_k$ are updated after every step, conditionally on the current state of the network. The corresponding proofs are standard, and are given in Section \ref{sec:furtherproofs}.

\begin{proposition}
\label{prop: updatingpq}
    For Model B, suppose that one of the following holds for 
    $(p_k,q_k,r_k)_{k \ge   m_0}$:
    \begin{enumerate}
        \item for all $k \ge   m_0$, we have $r_{k}\ge 1/2$, and conditionally on $G_k$,
    \begin{equation} \label{eq:qm}
        q_{k} \ge   \bbclr{1+\frac{2r_{k}-1}{2}\frac{N_{k,0}}{N_{k}}}^{-1} 
    \end{equation}
    so that $q_k\ge 2/3$ almost surely for all $k\ge m_0$; and if $r_{k}<1,$ further 
    \begin{align} \label{eq:pm}
      0\le   p_{k}
            \le \frac{2r_{k}-1}{1-r_{k}}\frac{N_{k,0}}{N_{k}}-\frac{2}{(1-r_{k})}\frac{1-q_{k}}{q_{k}}.
    \end{align}
   \item for all $k \ge   m_0$, $r_{k}\ge 1/2$, and conditionally on $G_k$, 
   \begin{align}\label{eq:qmrm}
        q_{k} \ge   \bbclr{1-\frac{1}{2}\frac{N_{k,0}}{N_{k}}}\bbclr{1-(1-r{_k})\frac{N_{k,0}}{N_{k}}}^{-1};
   \end{align}
   so that $q_k\ge 1/2$ almost surely for all $k\ge m_0$; and if $r_{k}<1$, further 
   \begin{align}\label{eq:pm1}
     0\le   p_{k}\le \bbclr{\bbclr{1-2q_{k}(1-r_{k})}\frac{N_{k,0}}{N_{k}} - 2(1-q_{k})}{\bclr{{q_{k}(1-r_{k})}}^{-1}}.
   \end{align}
    \end{enumerate}
    Then, in either cases, $N_{m,0}$ is a supermartingale in $L^1$ and converges almost surely as $m\to\infty$. Moreover as $m\to\infty$, we have $N_{m,0}/N_m\to 0$ in probability. 
\end{proposition}

In the assumptions of Proposition \ref{prop: updatingpq}, the requirement $r_k\ge 1/2$ is to ensure that $q_k$ can take value less than one. That  $q_k$ in \eqref{eq:qm} and \eqref{eq:qmrm} is larger than 2/3 and 1/2, respectively, can be seen by applying the inequalities $N_{k,0}/N_k\le 1$ and $r_k\le 1$. Under either sets of assumptions, if ${N_{k,0}}=0$ at some step $k \ge   m_0$, we have $q_{k}=1$, so each step is with probability one a duplication step. In this situation, no new isolated vertices can be created at the subsequent steps: this is clear if $r_k=1$, and if $r_k<1$, $p_k=0$ also, any edges created at the divergence phase are retained with probability $1-p_k=1$. In particular, if the initial network does not contain any isolated vertices, then $N_{m,0}=0$ almost surely for all $m \ge   m_0$. 

In Proposition \ref{prop:modelc2} 
the sequences $(p_k,q_k,r_k)_{k\ge  m_0}$ 
do not depend on the network growth process. The proposition
guarantees that if the expected number of deletion steps is bounded as $m\to\infty$ and the sequence $(p_m(1-r_m))_{m\ge m_0}$ has a finite sum, then the expected number of isolated vertices is also bounded in $m$. The proof is found in Section \ref{sec:furtherproofs}. 

\begin{proposition}\label{prop:modelc2}
    For Model B, suppose that the sequence $(p_k,q_k,r_k)_{k \ge   m_0}$ is such that 
     for all $k\ge m_0$, $r_k \ge   1/2$ or $q_k \le (2 (1-r_k))^{-1}$ (or both), and  \begin{equation}\label{eq:pqr}
        \sum_{k \ge   m_0}\{
        p_k(1-r_k)+ 2(1-q_k)\}<\infty.
    \end{equation}
    Then, {$N_{m,0}$ is bounded in expectation with respect to $m$, and $N_{m,0}/m\to 0$ in probability as $m\to\infty$.} 
    Moreover, $N_{m,0}/N_m\to 0$ in probability as $m\to\infty$.
\end{proposition}

Let $\alpha_1,\alpha_2>1$. An example
 of $(p_k,q_k,r_k)$ that satisfies the condition of Proposition \ref{prop:modelc2} is $p_k=k^{-\alpha_1}$, $q_k=1-k^{-\alpha_2}$ and $r_k\ge 1-(2q_k)^{-1}$. Another example is $q_k=1-k^{-\alpha_1}$, $r_k=\max\{1/2, 1-k^{-\alpha_2}\}$, and $p_k$ in this case is allowed to be arbitrary in $[0,1]$.

\begin{figure}[!tpb]
\centering 
   \includegraphics[width=0.85\textwidth, height=0.6\textwidth]{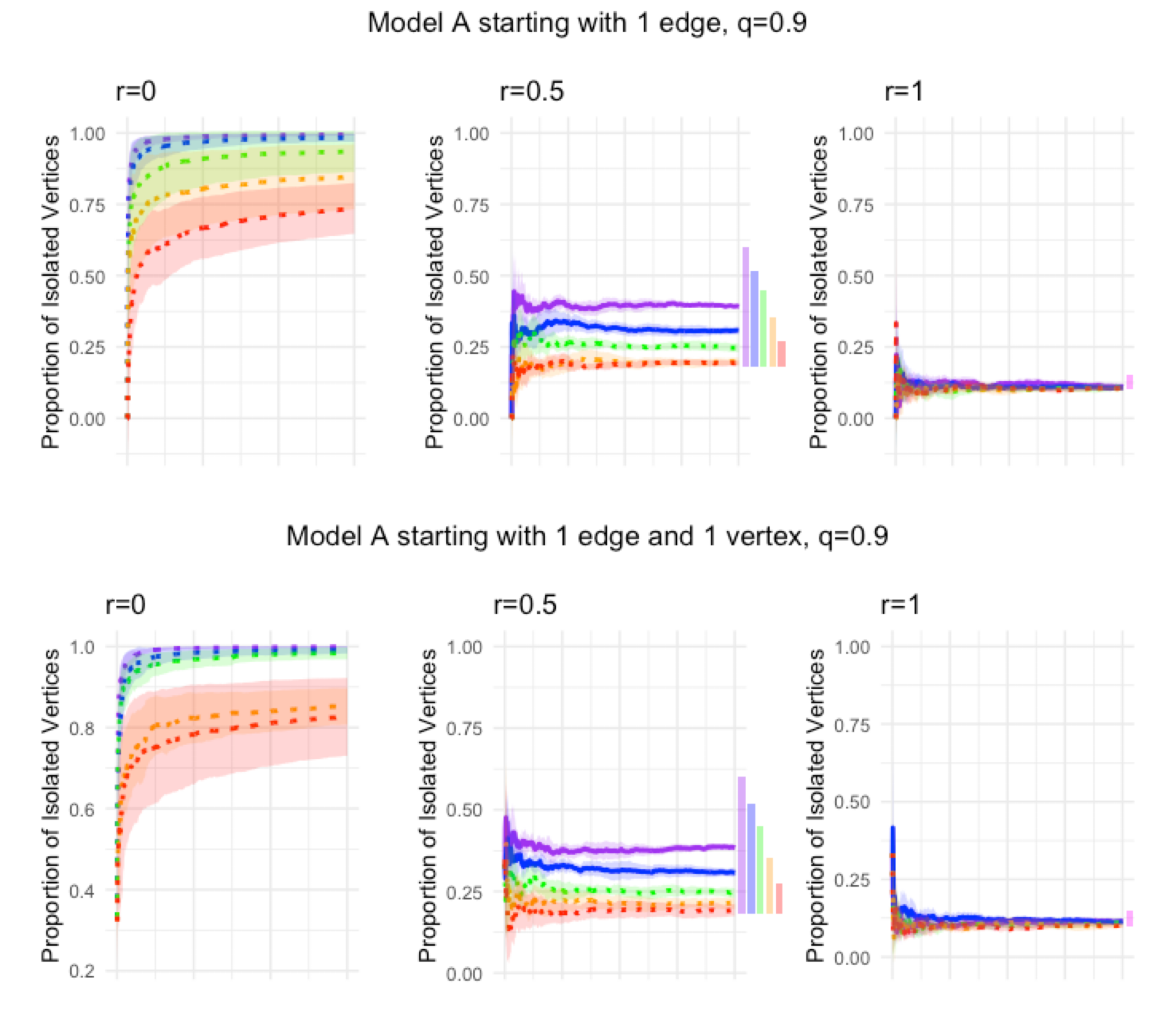} 
  \caption{\footnotesize Evolution of the average proportion of isolated vertices of Model A over time, initialised either with a single edge (top) or a combination of a single vertex and an edge (bottom). The simulations are run for 1500 steps when $r=0$, and 1000 steps when $r \neq 0$, as the rate of convergence for the proportion of isolated vertices appears to increase as $r$ increases. Different values of $p$ are represented by different colours: 
  \textbf{red}: $p=0$; \textbf{orange}: $p=0.2$; \textbf{green}: $p=2-\sqrt{22}/3\approx0.4365$; \textbf{blue}: $p=0.6$; \textbf{purple}: $p=0.8$. For $r = 0$, Theorem \ref{th:LBk}  requires $q \leq 1/2$. For $r \neq 0$ and $q = 0.9$, Theorem \ref{th:LBk} requires $p > 2 - \sqrt{22}/3$, subject to $q \leq \min \{1, 1/(2(1-r))\}$. Thus, the \textbf{solid} lines represent parameter sets satisfying the conditions of Theorem \ref{th:LBk}, while the \textbf{dashed} lines correspond to cases where these conditions are not met. Shaded regions correspond to $\pm 2$ standard errors; capped at $0$ and $1$. The coloured bars on the right indicate the intervals $[\rho_0,(\rho_1\wedge 1)]$, where $\rho_0$ in \eqref{de:rho} does not depend on $p$. When $r=1$, $\rho_1$ in \eqref{rho1} no longer depends on $p$, so the interval $[\rho_0,\rho_1]$ is represented by a single bar in magenta. 
} 
\label{fig:modelA}
  \vspace{-4mm}
\end{figure}

Before embarking on the proofs we present simulations for Model A with edge deletion. Each simulation is run for at least 1000 steps, repeated 30 times, and we plot the average proportion of isolated vertices. The initial graph is configured either as a single edge or as a single vertex combined with an edge. We primarily illustrate the case where $q = 0.9$, though similar patterns are observed for other values of $q \in [0, 1]$, with faster convergence as $q$ decreases (see \textbf{the Appendix}). The code for this simulations can be found at \url{https://github.com/rh-zhang/DD_Edge-Deletion}.

Figure \ref{fig:modelA} shows that in Model A, there are values of $p, q, r$ such that the proportion of isolated vertices appears to converge to a limit between 0 and 1, unaffected by the initial configuration. Notably, the limiting proportion of isolated vertices increases with $p$, and reaches 1 in some instances. The panels for the $r=1$ case illustrate the second result of Corollary \ref{cor}; noting that when $q=0.9,r=1$, $[\rho_0,\rho_1]=[0.1,0.15]$.  
Simulations also suggest that Theorem \ref{th:LBk} holds even in some cases where its conditions are not met: for instance, the proportion of isolated vertices seems to converges to 1 when $p=0.8, q=0.9, r=0$. This suggests that the theorem could potentially be extended to broader conditions.  

\section{Recursion Formulas}\label{sec:recursion}
The analysis of the two models in this paper relies heavily on recursion formulas for the expectations. First we state the recursion formulas for $N_{m,k}$ in both models. Then, we show that in Model~A, $\E[(N_{m,k}-N_{m-1,k})^2]$ can be bounded in terms of the second moment of the expected degree distribution, a finding which underpins our martingale argument in Section \ref{section: convergence0}, and we give an upper bound for the latter quantity.  
The proofs are straightforward, and are deferred to
Section \ref{sec:furtherproofs}.

\begin{proposition} \label{prop:recursion}
The recursion formulas for $N_{m,k}$, $k\ge 1$ in both models are follows.
\begin{enumerate}
\item
    For  Model A, we have
    \begin{multline}
\label{eq:new_nm0_diff}
    \mathbb{E}( N_{m, 0}-N_{m-1,0} \mid G_{m-1}) =  q(1-2r)\frac{N_{m-1,0}}{N_{m-1}}+q(1-r)\sum^{m-2}_{j =    1}\frac{N_{m-1,j}}{N_{m-1}}p^{j}\\ + (1-q) \left(2+\frac{N_{m-1,1}-N_{m-1,0}}{N_{m-1}} \right),
\end{multline}
whereas for $k \ge  1,$
\begin{align}
\label{eq:new_nmk_diff}
    &\mathbb{E}( N_{m, k} - N_{m-1,k} \mid G_{m-1}) = \frac{N_{m-1,k-1}}{N_{m-1}}q\big(r+r(1-p)^{k-1}+(1-p)(k-1)\big) \nonumber\\
    &\quad -\frac{N_{m-1,k}}{N_{m-1}}q(r+(1-p)k)+  q(1-r) \sum^{m-2}_{j = k}\frac{N_{m-1,j}}{N_{m-1}}{j \choose k}p^{j-k}(1-p)^k \nonumber \\ 
    &\quad + qr\sum^{m-2}_{j = k}\frac{N_{m-1,j}}{N_{m-1}} {{j} \choose {k-1}}p^{j-k+1}(1-p)^{k-1} + (1-q) (k+1)\frac{N_{m-1,k+1}-N_{m-1,k}}{N_{m-1}}.
\end{align}

\item 
For  Model B, \eqref{eq:new_nmk_diff} holds with $p,q,r$ replaced with $p_{m-1},q_{m-1},r_{m-1}$, and instead of \eqref{eq:new_nm0_diff} we have 
\begin{multline}
\label{eq:new_nm0_diff3}
    \mathbb{E}( N_{m, 0}-N_{m-1,0} \mid G_{m-1}) =  q_{m-1}(1-2r_{m-1})\frac{N_{m-1,0}}{N_{m-1}}+q_{m-1}(1-r_{m-1})\sum^{m-2}_{j =    1}\frac{N_{m-1,j}}{N_{m-1}}p_{m-1}^{j}   \\
     + (1-q_{m-1}) \left(1+\frac{N_{m-1,1}-N_{m-1,0}}{N_{m-1}} \right).
\end{multline}
\end{enumerate}
\end{proposition}

\begin{proposition} \label{prop:isolatedbound} For Model A, 
we have
\begin{align}
    \label{eq:isolatedbd}
    \E[(N_{m,0}-N_{m-1,0})^2\mid G_{m-1}] & \le 5(1-q)\sum^{m-2}_{j= 0}j^2 \frac{N_{m-1,j}}{N_{m-1}}  + 5;
\\
\label{eq:sqdiffk}
    \E[(N_{m,k}-N_{m-1,k})^2\mid G_{m-1}
      ]
 &\le (3+2q) \sum_{j=0}^{m-2} j^2 \frac{N_{m-1, j}}{N_{m-1}} + 4,  \quad \quad \mbox{ for } k \ge 1. 
  \end{align}
\end{proposition}

In view of Proposition \ref{prop:isolatedbound}, an upper bound on the second moment of the expected degree distribution at time step $m-1$ immediately yields an upper bound on $\E[(N_{m,k}-N_{m-1,k})^2]$. 
We denote  $k^{th}$ moment of the expected degree distribution at time $m$ for Model A by
\begin{align}
    \label{de:psi}
    \psi_k(m) =  
     \sum_{j=0}^{m-1} j^k \frac{   \E N_{m, j}}{m}. 
\end{align}

\begin{lemma} \label{lemmafirstmoment}
 Let 
   $ \psi_j(m)
   $ be as above, $ 
  \tau = 6q-4pq+p^2q-3$ and $\kappa = 4 q - 2 p q - 2$ be as in \eqref{tau} and \eqref{kappa}.
   Then, for Model A, there is a constant $C>0$ such that
   \begin{align*}
    \psi_1(m)\le 
    \begin{cases}
        C &\text{if $ \kappa <0$}\\
        C\log m,&\text{if $ \kappa =0$}\\
        Cm^{\kappa},&\text{if $ \kappa >0$}. 
\end{cases}\numberthis\label{eq:psi1bd}
\end{align*} Similarly, for $\kappa<0$, 
\begin{align}\label{b1}
    \psi_2(m) \le 
    \begin{cases}
        C&\quad \text{if $\tau <0$,}\\
        C\log m&\quad \text{if $ \tau =0$,}\\
        Cm^{\tau }&\quad \text{if $\tau >0$.}
    \end{cases}
\end{align}
For $\kappa=0$, 
\begin{align}\label{b2}
    \psi_2(m) \le 
    \begin{cases}
        C\log m&\quad \text{if $\tau<0$,}\\
        C\log^2 m&\quad \text{if $ \tau=0$,}\\
        Cm^{\tau }&\quad \text{if $ \tau>0$,}
    \end{cases}
\end{align}
Finally, for $\kappa>0$,  
\begin{align}\label{b3}
    \psi_2(m) \le  \begin{cases}
        Cm^\kappa &\quad \text{if $\tau\le 0$,}\\
        Cm^{\max\{\tau,\kappa\}}&\quad \text{if $ \tau>0$ and $\tau\ne \kappa$,}\\
        Cm^\tau \log m&\quad \text{if $ \tau>0$ and $\tau= \kappa$.}
    \end{cases}
\end{align} 
\end{lemma}

\begin{remark}\label{tk}
      Note that $\kappa$ {in \eqref{kappa}} is monotone increasing in $q$. So if $\tau<1$, then $ q<4/(6-4p+p^2),$
    which in turn implies that
   $
       \kappa = (4-2p)q-2<{(16-8p)}/{(6-4p+p^2)}-2\le  2\sqrt{2}-2\approx 0.8284;
  $  
    here we used
    that $(16-8p)/(6-4p+p^2)$ is maximised at $p=2-\sqrt{2}\approx 0.5858$. Thus, any pair $(p,q)\in[0,1]$ satisfying $\tau<1$ also satisfies  $\kappa<1$; see Figure \ref{fig:pqr} for $(p,q)$ such that $\tau<1$. 
\end{remark} 

\begin{remark}  
Analogues of Proposition \ref{prop:isolatedbound} and Lemma \ref{lemmafirstmoment} can be obtained similarly for Model B, when $p_m=p$, $q_m=q$ and $r_m=r$ for all $m\ge m_0$. 
For conciseness we do not delve into these details.
\end{remark}

\section{A Martingale Tool}\label{sec:backhausz}
A key tool for proving Theorems \ref{th:LBk} and \ref{th:LBK1} is a modified version of 
Lemma 1 in  \cite{Backhausz}, as stated below. 

\begin{lemma}
\label{lemma:weakBM}
    Given $m_0 \ge  0$, {$\xi_{m_0}\geq 0$}, let $(\xi_m)_{m \ge    m_0}$ be a non-negative {integrable} process adapted to the filtration $(\mathcal{F}_m)_{m \ge    m_0}$, with $\cF_{m_0}$ being the trivial $\sigma$-algebra. Assume that for all $m\ge m_0+1$, 
 \begin{align}\label{eq:genconexp}
        \E[\xi_m\mid\mathcal{F}_{m-1}] \ge    \bbclr{1-{\frac{u}{m-1}} }\xi_{m-1} + v_m
    \end{align}
   for some $0\le u\le m_0$, and  $(v_m)_{m \ge   m_0}$ is a non-negative predictable process such that almost surely, for some constant $v\ge  0$,
 \begin{equation}\label{eq:v}
    v_m \ge  v, \quad m\geq m_0. 
    \end{equation}
    Suppose also there are constants $\widetilde C,\delta>0$ not depending on $m$ such that 
    \begin{equation}\label{eq:sdc}
        \E[(\xi_m-\xi_{m-1})^2]\le \widetilde Cm^{1-\delta},\quad m\ge m_0
        .
    \end{equation}
   Let
   \begin{align}
       \alpha_\delta &:=\mathbbm{1}[\delta<2]\cdot \delta/4 + \mathbbm{1}[\delta\ge 2]\cdot 1/2; \label{alpha}\\
          \beta_{m,\delta} &:= \mathbbm{1}[\delta<2] \cdot m^{-\delta/2} + \mathbbm{1}[\delta=2] \cdot m^{-1}\log m + \mathbbm{1}[\delta>2] \cdot m^{-1}\label{beta}.
   \end{align}
    Then there exists a constant $C=C(m_0,u)>0$ 
    such that with probability at least $1-C\beta_{m,\delta}$, 
\begin{equation}\label{eq:lb}
       \frac{\xi_m}{m} \ge   \frac{v}{u+1}-Cm^{-\alpha_{\delta}}.  
    \end{equation}
If \eqref{eq:genconexp} holds with strict equality, and instead (or in addition) of \eqref{eq:v}, there is a constant $V >0$ such that 
\begin{equation}\label{eq:V}
    v_m \le  V, \quad m\geq m_0,  
    \end{equation}
then there exists a constant $C=C(\wt C, m_0,u, \delta)>0$ such that with probability at least $1-C\beta_{m,\delta}$, 
\begin{equation}\label{eq:vmb}
       \frac{\xi_m}{m} \le   \frac{V}{u+1}+Cm^{-\alpha_{\delta}}.
    \end{equation}
\end{lemma}

\begin{remark}
   For Model A, the inequality \eqref{eq:sdc}, 
   with $\xi_{m}=N_{m,k}$,   is much easier to verify than the analogous condition set out in Lemma 1 of \cite{Backhausz}, which is
\begin{equation}\label{sqbd}
    \mathbb{E}[( N_{m, k} - N_{m-1, k})^2\mid  \mathcal{F}_{m-1}]\le O\big(m^{1-\delta}\big)\quad \text{for some $\delta>0$.}
\end{equation}
The trade-off is the weaker conclusion than that obtained in  \cite{Backhausz}, where \eqref{eq:lb} is replaced with  $\liminf_{m\to\infty} \xi_m/m\ge v/(u+1)$ almost surely (and similarly for the upper bound). In the proof of Theorem \ref{thm: pq1case}, we show that \eqref{sqbd} can be easily verified at least when $p=q=1$.  We also note that $0 \le N_{m,k} \le m$ and hence $N_{m,k}$ is integrable. 
\end{remark}

\begin{proof}[Proof of Lemma \ref{lemma:weakBM}]
     Define $c_{m_0}=1$ and
    \begin{equation}\label{cm}
        c_m:=\prod^m_{i=m_0+1}\bbclc{1-\frac{u}{i-1 }}^{-1}= \frac{\Gamma(m)\Gamma(m_0-u)}{\Gamma(m-u)\Gamma(m_0)},\quad m\geq {m}_0,
    \end{equation}
    which is positive for all $m \ge   {m}_0$, since $u< i$ for all $i \ge   {m}_0$.
    By \eqref{eq:genconexp} and \eqref{cm}, for $m\geq m_0$, 
    \begin{align}\label{eq:cmxim}
         \mathbb{E} [ c_m \xi_m \mid \mathcal{F}_{m-1} ] \ge   c_m \left(1 - {\frac{u}{m-1}} \right) \xi_{m-1} + c_m v_m = c_{m-1} \xi_{m-1} + c_m v_m.
    \end{align}
    Since $v_m$ are non-negative,  $(c_m \xi_m{-c_{m_0}\xi_{m_0}})_{m \ge   m_0}$ is therefore a submartingale with respect to $\cF_{m-1}$. 
    Moreover, since $\xi_m$ is square integrable by \eqref{eq:sdc}, so is $c_m \xi_m$. Now, write the Doob decomposition of the submartingale $(c_m\xi_m{-c_{m_0}\xi_{m_0}})_{m \ge   m_0}$, so that 
\begin{equation}\label{eq:doobineq}
        c_m\xi_m{-c_{m_0}\xi_{m_0}} = M_m + A_m,
    \end{equation}
    where
        $ M_m=\sum^m_{i=m_0+1}(c_i\xi_i-\E[c_i\xi_i\mid\mathcal{F}_{i-1}])$,  with $ {M_{m_0} = 0}, $
    is a mean zero martingale and  
    $
        A_m = \sum_{i=m_0+1}^m (  \E[c_i\xi_i\mid\mathcal{F}_{i-1}]-c_{i-1} \xi_{i-1}), $ with $ {A_{m_0} = 0},  
  $
    is predictable with respect to $\cF_m$. By \eqref{eq:cmxim}, 
    \begin{equation}\label{AA}
        A_m-A_{m-1}= \E[c_m\xi_m\mid\mathcal{F}_{m-1}]-c_{m-1} \xi_{m-1}  \ge   c_m v_m \ge   0, \quad m\geq m_0+1,
    \end{equation}
    implying that $(A_m)_{m \ge    m_0}$ is a non-negative and (weakly) increasing. Moreover, a computation similar to  \cite[Lemma 13]{frieze2020degree} (taking $a=0$ and $b=-u$ there) yields
\begin{align}\label{sc}
    \sum^m_{i=m_0+1 }c_i = \sum^{m}_{i=m_0+1} \frac{\Gamma(m_0-u)\Gamma(i)}{\Gamma(m_0)\Gamma(i-u)} =\frac{\Gamma(m_0-u)}{\Gamma(m_0)} \frac{m^{u+1}}{u+1} +O(m^u).
\end{align}
Thus by \eqref{AA}, \eqref{eq:v} and \eqref{sc}, 
\begin{align}\label{eq:Alb}
     A_m \geq \sum^m_{i=m_0+1}c_iv_i \geq   v \sum^m_{i=m_0+1}c_i \geq  \frac{\Gamma(m_0-u)}{\Gamma(m_0)} \frac{vm^{u+1}}{u+1} +O(m^u).
\end{align}
We now show that when $m$ is large, $|M_m|$ is typically much smaller than $A_m$. Let $B_m$ be the increasing process associated with the Doob decomposition of the submartingale $M_m^2$. Since $M_m$ is a mean zero martingale, a straightforward computation gives
\begin{equation}\label{BM}
    B_m=\sum_{i=m_0+1}^m \big( \EE \big[M_i^2 | \cF_{i-1}\big] - M_{i-1}^2\big) =\sum_{i=m_0+1}^m  \EE \big[(M_i- M_{i-1})^2 | \cF_{i-1}\big]=\sum^m_{i=m_0+1}\var(c_i\xi_i\mid\mathcal{F}_{i-1}).
\end{equation}
Thus,
\begin{align*}
B_m =  \sum^m_{i=m_0+1}\var(c_i(\xi_i-\xi_{i-1})\mid\mathcal{F}_{i-1})
& \le D_m:= \sum^m_{i=m_0+1}c^2_i\E[(\xi_i-\xi_{i-1})^2\mid \mathcal{F}_{i-1}].
    \end{align*}
Taking expectation and using the fact that $c_m$ is increasing,
    \begin{align}\label{D}
         \E D_m &= \sum^m_{i=m_0+1} \E[c_i^2(\xi_i-\xi_{i-1})^2]  
        \le \sum^m_{i=m_0+1} c_m^2 \E[(\xi_i-\xi_{i-1})^2].
    \end{align}
Below, $C>0$ is a constant that may vary, and depends on $\wt C$, $m_0$, $u$ and $\delta$. By applying \eqref{eq:sdc} to \eqref{D}, 
\begin{align}\label{eq:dm}
    \E D_m \le c_m^{2}\sum^m_{i=m_0+1}        \widetilde C i^{1-\delta}\le 
  \begin{cases}
      C m^{2-\delta}c_m^2&\text{if $\delta< 2$;}\\
      \widetilde C(\log m)c_m^2& \text{if $\delta= 2$;}\\
      C 
     m_0^{2-\delta} c_m^2 & \text{if $\delta> 2$.}
  \end{cases} 
\end{align}
Let $d_m:=m^{1-\alpha_\delta}c_m$, 
and note $d_m=m^{1+u-\al_\delta}(1+O(m^{-1}))$. By \eqref{BM}, $\E B_m = \E M_m^2$.
So by Markov's inequality, \eqref{alpha}, \eqref{eq:dm}, and separately  considering  the cases, 
\begin{align}\label{m2}
    \IP(|M_m| \ge d_m) =  \IP(M^2_m\ge d_m^2)\le \frac{\E M_m^2}{d_m^2} = \frac{\E B_m}{d_m^2} \leq \frac{\E D_m}{d_m^2} \le \begin{cases}
        Cm^{-\delta/2}&\text{if $\delta<2$;}\\
       \widetilde  Cm^{-1}\log m&\text{if $\delta= 2$;}\\
        Cm^{-1}&\text{if $\delta> 2$.}
    \end{cases}
\end{align}
Since $c_{m_0}\xi_{m_0}\ge 0$, by \eqref{eq:doobineq} we also have
\begin{align}\label{eq:fd}
     c_m\xi_m \geq {A_m} - |M_m|.
\end{align} 
In view of \eqref{eq:Alb} and \eqref{m2}, we can deduce that there is a constant $C>0$ such that with probability at least $1-C m^{-\beta_{m,\delta}}$,
\begin{align}\label{Amm}
     A_m - |M_m| \geq 
 \frac{\Gamma(m_0-u)}{\Gamma(m_0)} \frac{v m^{u+1}}{u+1} -d_m + O\big(m^{u}\big). 
\end{align}
Furthermore, by Stirling's formula,
$
        c_m = \frac{\Gamma(m_0-u)}{\Gamma(m_0)} m^u \big(1+O\big(m^{-1}\big)\big); 
$
combining this with \eqref{eq:fd} and \eqref{Amm}, and noting that $0<\alpha_\delta \le 1/2$,
\begin{align*}
   \frac{\xi_m}{m} \ge \frac{A_m-|M_m|}{m c_m} \geq \frac{v}{u+1} - C m^{-\alpha_{\delta}},
\end{align*}
which is \eqref{eq:lb}.

Similarly, assume that instead of (or in addition to) \eqref{eq:v}, we have \eqref{eq:V}, and \eqref{eq:genconexp} holds with equality.
 Then, $(c_m\xi_m-c_{m_0}\xi_{m_0})_{m\ge m_0}$ is still a submartingale.
Thus, using $ c_m \xi_m \leq A_m + |M_m|+  c_{m_0} \xi_{m_0}, $ 
we have from \eqref{m2} that 
there is a constant $C>0$ such that with probability at least $1-C m^{-\beta_{m,\delta}}$,
\begin{align*}
c_m \xi_m \leq A_m + d_m + c_{m_0} \xi_{m_0} .
\end{align*}
Furthermore, it follows from \eqref{eq:doobineq} that $A_m-A_{m-1}=c_mv_m$, and $\eqref{eq:V}$ further implies that $A_m = \sum_{i=m_0}^m c_i v_i \le V  \sum_{i=m_0}^m c_i.$ 
 Hence the assertion \eqref{eq:vmb} follows from a computation similar to \eqref{eq:Alb}.
\end{proof}

\section{Proofs of Theorems \ref{lem:meanbounds}, \ref{th:LBk} and \ref{th:LBK1}}
\label{section: convergence0}
With the previous preparation, we now set out to prove the main results. We take
 \begin{align}\label{eq:vm}
     u= 1+2q(r-1) \quad \text{and} \quad 
        v_m= (1-q)\bbclr{2+\frac{N_{m-1,1}}{m-1}} + q(1-r)\sum^{m-2}_{j=1}\frac{N_{m-1,j}}{m-1}p^j, 
    \end{align} 
and note that almost surely, for all $m\ge m_0$,
\begin{align}\label{vV}
        v:=2(1-q) \le v_m \le V:=3 (1-q) + pq(1-r).
    \end{align}
\begin{proof}[Proof of Theorem \ref{lem:meanbounds}]
Taking expectations in \eqref{eq:new_nm0_diff} yields 
$$\EE N_{m,0} =  \EE N_{m-1,0} 
        \left( 1 - \frac{u}{m-1} \right) 
        + v_m \le \EE N_{m-1,0} 
        \left( 1 - \frac{u}{m-1} \right) 
        + V.$$
Dividing by $m$ and taking the upper bound gives a recursion of the form 
$g(m+1) \le g(m) \left( 1 - \frac{u}{m}\right) + \frac{V}{m},$
where $g(m)=\E(N_{m,0})/m$. This is equivalent to 
$ g(m+1) - \frac{V }{u} \le \left( 1 - \frac{u}{m}\right) \left(g(m) - \frac{V }{u} \right) , $
which has as solution 
\begin{equation*}
    g(m+1) - \frac{V }{u} \le  \left(  g(m_0) - \frac{V }{u}\right)\prod^m_{i=m_0}  \left( 1 - \frac{u}{m}\right) = O(m^{-u}),
\end{equation*}
where $u>0$ by assumption.
The stated upper bound on $\E(N_m)/m$ thus follows immediately. The lower bound can be proved similarly using $v$  in \eqref{vV} and we therefore omit the details.
\end{proof}

The next proof again relies on the recurrence formulas.

\begin{proof}[Proof of Theorem \ref{th:LBk}]
    We prove the lower bound claimed in \eqref{eq:lb} by applying Lemma \ref{lemma:weakBM} to the recursion formula \eqref{eq:new_nm0_diff} for $N_{m,0}$, taking $u$ and $v_m$ as in \eqref{eq:vm}, and $\xi_{m,0} = N_{m,0}$.
    Then noting that $N_{m} = m$ for Model A, \eqref{eq:new_nm0_diff} reads 
     \begin{align*}
    \mathbb{E}( \xi_{m, 0} \mid G_{m-1}) 
        =   \xi_{m-1,0} 
        \left( 1 - \frac{u}{m-1} \right) 
        + v_m.
\end{align*}

First we verify the conditions for applying Lemma~\ref{lemma:weakBM}. 
Observe that $u \geq 0$ when $r = 1$ or $r < 1$ with $q \leq 1/(2(1-r))$, and at least one of these conditions holds under the theorem. Additionally, we have that  $u \le 1 <  m_0$ for all $m \geq m_0\ge 2$; where $m_0 \ge 2$ by assumption. We choose $p,q$ such that $\tau<1$,  so that $\kappa=4q-2pq-2<1$; see Remark \ref{tk}. By \eqref{de:psi} and combining the bounds in Lemma~\ref{lemmafirstmoment} for different values of $\tau,\kappa$, there is a constant $C>0$ such that 
\begin{align*}
 \psi_2(m) &\leq C \log ^2 m \cdot \mathbbm{1}[\tau\le 0,\kappa\le 0]  + Cm^{\max\{\tau,\kappa\}} \log m \cdot\big(1-\mathbbm{1}[\tau\le 0,\kappa\le 0]\big). 
\end{align*} 
Thus, by the last display, Proposition~\ref{prop:isolatedbound}, and choosing $0<\delta<\min\{1,1-\max\{\tau,\kappa\}\}$, for $m$ large enough, $ \E[(\xi_{m,0}-\xi_{m-1,0})^2]\le Cm^{1-\delta}$.
Furthermore, almost surely, \eqref{vV} holds for all $m\ge m_0$.
  Thus,  the conditions set out in Lemma \ref{lemma:weakBM} are met, including \eqref{eq:genconexp} with strict equality, and with $\alpha_\delta=\delta/4$ and $\beta_{m,\delta}=m^{-\delta/2}$ in \eqref{alpha} and \eqref{beta}. It follows 
   that there exists a constant $C>0$ such that with probability at least $1-Cm^{-\delta/2}$,
\begin{align*}   
    \frac{v}{u+1}-C m^{-\delta/4} \le     \frac{N_{m,0}}{m}  \le   \frac{V}{u+1}+C m^{-\delta/4},  
\end{align*}  
with $v/(u+1)=\rho_0$ and $V/(u+1)=\rho_1$, as in \eqref{de:rho} and \eqref{rho1}.
\end{proof}

\begin{proof}[Proof of Theorem \ref{th:LBK1}]
The proof is similar to the one of Theorem \ref{th:LBk}.  
When $p=1$, then we have that $\sum_{j\geq 1} p^j N_{m,j}/m=1-N_{m,0}/m$ and so the recursion formula \eqref{eq:new_nm0_diff} simplifies to
\begin{align*}
    \E(N_{m,0}\mid G_{m-1})
    &= N_{m-1,0}\bbclr{1+\frac{q(1-2r)-q(1-r)-(1-q)}{m-1}} + q(1-r) + (1-q)\bbclr{2+\frac{N_{m-1,1}}{m-1}}.
\end{align*} 
We take $\xi_{m} = N_{m,0}$ as before, and
\begin{eqnarray*}
        u' &=& q(1-r)+(1-q)-q(1-2r)=1-q+qr,\\
        v'_{m}&=& q(1-r) + (1-q)\bbclr{2+\frac{N_{m-1,1}}{m-1}}.
    \end{eqnarray*}
Clearly, $u'\ge 0$ for all $q,r\in[0,1]$, and $u'\le 1<m_0$. We have, almost surely for all $m\ge m_0$,
\begin{align*}
   v':=q(1-r)+2(1-q) \le v'_{m}\le V':= q(1-r) + 3(1-q) . 
\end{align*}
Furthermore if $p=1$, then $\tau=3q-3\le 0$ and $\kappa=2q-2\le 0$. By Lemma \ref{lemmafirstmoment}, $\psi_2(m)\leq C\log^2 m$ and so by Proposition \ref{prop:isolatedbound}, $\E[(\xi_{m,0}-\xi_{m-1,0})^2]\le Cm^{1-\delta}$, with $0<\delta<1$. Hence, the conditions for applying Lemma \ref{lemma:weakBM} hold and we conclude that there is a constant $C>0$ such that with probability at least $1-Cm^{-\delta/2}$, 
\begin{align}
 \frac{v'}{1+u'} -C m^{-\delta/4} \le   \frac{N_{m,0}}{m} \le \frac{V'}{1+u'} -C m^{-\delta/4},
\end{align}
with $v'/(u'+1)=\theta_0$ and $V'/(u'+1)=\theta_1$, as in \eqref{theta} and \eqref{theta1}. 
\end{proof}

\section{Further Proofs} \label{sec:furtherproofs}

\begin{proof}[Proof of Proposition \ref{prop:modelB}]
Assume that $p=0$, $q=1$ and $0\le r \le 1$.
 Let $\alpha_{m_0}  =1$ and
    \begin{equation}\label{de:al}
        \alpha_m = \prod^{m-1}_{k=m_0}\bbclr{1+\frac{1-2r}{k}}^{-1}=\frac{\Gamma(m)\Gamma(m_0+1-2r)}{\Gamma(m+1-2r)\Gamma(m_0)},\qquad m \ge   m_0+1. 
    \end{equation}
We first show that $\alpha_m N_{m,0}$ is a martingale in $L^1$.
     As $p =0$ and $q=1$, 
     the recursion formula \eqref{eq:new_nm0_diff} for Model A reduces to 
     \begin{align*}
        \E(N_{m,0}\mid G_{m-1}) &= N_{m-1,0} \bbclr{1+\frac{1-2r}{m-1}}, 
        \label{eq:Nm0p0q1modelB}
    \end{align*}
    which, with $\al_m$ as in \eqref{de:al}, implies that
    \begin{align*}
        \alpha_m\E(N_{m,0}\mid G_{m-1}) =  N_{m-1,0} \prod^{m-2}_{k=m_0}\bbclr{1+\frac{1-2r}{k}}^{-1} = \alpha_{m-1} N_{m-1,0}.
    \end{align*}
    Note that $\alpha_m$ is well-defined, {also} in the case of $m_0=1$, 
    and $\alpha_m>0$ for all $m \ge    m_0$. Thus, we can conclude that for the filtration $\cF_m = \sigma(G_m), m \ge  m_0,$ the sigma-fields generated by the graphs, $\alpha_{m} N_{m,0}$ is a non-negative martingale with 
    \begin{equation}\label{mmean}
         \E(\alpha_mN_{m,0}) = \alpha_{m_0} N_{m_0,0} = N_{m_0,0}.
    \end{equation}
    It follows from the martingale convergence theorem that $\alpha_mN_{m,0}$ has an almost sure limit as $m\to\infty$. Since $\alpha_m= m^{-(1-2r)}(1+O(m^{-1}))$ by Stirling's formula, we also conclude that $m^{-(1-2r)}N_{m,0}$ converges almost surely as $m\to\infty$. For the second assertion, we have by \eqref{mmean} 
    \begin{align}\label{eq:MGmean}
         \E(N_{m,0}) \le  C m^{1-2r}
    \end{align}
    for some constant $C>0$. Thus if $r\ge 1/2$, $N_{m,0}$ is bounded in  expectation as $m\to\infty$; and if $r>0$, it follows from Markov's inequality and \eqref{eq:MGmean} that
     $N_{m,0}/m\to 0$ in probability as $m\to\infty$. 
\end{proof}

\begin{proof}[Proof of Theorem \ref{thm: pq1case}]

Here we assume $p=q=1$.  Each new vertex is almost surely isolated if $r=0$, and so $\lim_{n\to\infty}N_{m,0}/m=a_0=1$ almost surely. Thus, we assume now $r>0$ and show that for all $k\ge0$,
\begin{align}\label{eq:liminfpq1}
    \liminf_{m\to\infty} {N_{m,k}}/{m}\geq a_k\quad \text{almost surely.}
    \end{align}
    Once we show this and $\sum_{k\ge 0}a_k=1$, \eqref{eq:pq1lim} follows immediately from \cite[Lemma 3.1]{Thornblad2014}, which indicates that if \eqref{eq:liminfpq1} holds and $\sum_{k\ge 0} a_k=1$, then $ \lim_{m\to\infty} N_{m,k}/m= a_k$ almost surely. 
    
    To prove \eqref{eq:liminfpq1}, we adapt the induction argument of \cite{Thornblad2014},
   which in turn hinges on  \cite[Lemma 1]{Backhausz}. 
   We start by claiming that when $p=q=1$,
    \begin{align}\label{eq:pq11}
        \E[(N_{m,k}-N_{m-1,k})^2\mid G_{m-1}]\leq \begin{cases}
            4,&k=1;\\
            1,&k\ne 1,
        \end{cases}
    \end{align}
    so that condition \eqref{sqbd} for applying \cite[Lemma 1]{Backhausz} is satisfied.
    To verify \eqref{eq:pq11}, note that the vertex degrees are non-decreasing. The count of isolated vertices changes by at most one per step, either by adding a new isolated vertex {or by an existing one gaining a neighbour.} The count of vertices of degree~$k$, $k \geq 2$, changes by at most one if a vertex of degree $k-1$ or $k$ is chosen for attachment. The count of leaves increases by one if the new vertex connects to a non-isolated vertex, and by two if it connects to an isolated vertex. This concludes the proof of \eqref{eq:pq11}.
    
    We next use an induction on $k$. For $k=0$, the recursion formula in \eqref{eq:new_nm0_diff} or a direct calculation gives 
    \begin{align*}
        \E[N_{m,0}\mid G_{m-1}]= N_{m-1,0}\bbclr{1-\frac{r}{m-1}} + (1-r).
    \end{align*}
    Define $\xi_{m,0}=N_{m,0}$, $ u_{m,0}=(mr)/(m-1)$ and $v_{m,0} = 1-r$.
    Clearly, $u:= \lim_{m\to\infty} u_{m,0} = r$ and $ v_0:=\lim_{m\to\infty} v_{m,0}=1-r$. Thus,
    by \cite[Lemma 1]{Backhausz},
    \begin{align*}
        \liminf_{m\to\infty} \frac{N_{m,0}}{m}\geq \frac{v_0}{1+u}=
        \frac{1-r}{1+r}=:a_0 \quad \text{almost surely}.
    \end{align*}
    Now, suppose that \eqref{eq:liminfpq1} holds for 
    all $0\le j \le k-1$. By a direct computation, or \eqref{eq:new_nmk_diff} with $p=q=1$, 
    \begin{align*}
        \E[N_{m,1}\mid G_{m-1}]&= N_{m-1,1} \bbclr{1-\frac{r}{m-1}} + r\bbclr{1+\frac{N_{m-1,0}}{m-1}};\\ 
         \E[N_{m,k}\mid G_{m-1}]&= N_{m-1,k} \bbclr{1-\frac{r}{m-1}} + r\frac{N_{m-1,k-1}}{m-1},\qquad k\geq 2.
    \end{align*}
    We take $\xi_{m,k}=N_{m,k}$, $u=r$ and $  v_{m,k} =  r\bclr{\mathbf{1}[k=1]+N_{m-1,k-1}/(m-1)}$ in the application of \cite[Lemma 1]{Backhausz}. 
    By the induction hypothesis, 
    \begin{align*}
        \liminf_{m\to\infty} v_{m,k} \ge v_k := r(\mathbf{1}[k=1]+a_{k-1}) \quad\text{almost surely.}
    \end{align*}
    Hence, by \cite[Lemma 1]{Backhausz} 
    and the definition of $a_k$ in \eqref{ak}, 
    \begin{align*}
        \liminf_{m\to\infty} \frac{N_{m,k}}{m}\geq \frac{v_k}{1+u} = \frac{r}{1+r} (\mathbf{1}[k=1]+a_{k-1})= \bbclr{\frac{r}{1+r}}^k (1+a_0)=:a_k.
    \end{align*}
    This completes the induction and therefore proves \eqref{eq:liminfpq1}. Furthermore, 
    \begin{align*}
        \sum_{k\ge 1} a_k = (1+a_0)\sum_{k\ge 1} \bbclr{\frac{r}{1+r}}^k = \frac{2}{1+r} \bbclr{\frac{1}{1-r/(1+r)}-1}= \frac{2r}{1+r},
    \end{align*}
    implying that
    $\sum_{k\ge 0} a_k= a_0+2r/(1+r)= 1$.
\end{proof}

\begin{proof}[Proof of Proposition \ref{prop: updatingpq}] We prove each result separately. Under the assumptions in Item 1, we apply the inequalities $N_{m-1,1}/N_{m-1}\le 1$ and $\sum_{j\ge 1}p^j_{m-1} N_{m-1,j}/N_{m-1}\le p_{m-1}$ to
 \eqref{eq:new_nm0_diff3} so that 
   \begin{align}\label{r1}
       \E[N_{m,0}\mid G_{m-1}]\le 
   N_{m-1, 0} - R_{m-1} 
 +(1-r_{m-1})p_{m-1}q_{m-1} 
 - (1-q_{m-1})  \frac{N_{m-1,0}}{N_{m-1}}, 
   \end{align}
where
   \begin{align*}
    R_{m-1}&:=(2r_{m-1}-1)q_{m-1}\frac{N_{m-1,0}}{N_{m-1}} -2(1-q_{m-1})=\bigg((2r_{m-1}-1)\frac{N_{m-1,0}}{N_{m-1}}+2\bigg)q_{m-1} -2.\end{align*}
For $q_m$ as in \eqref{eq:qm},     
\begin{align*}
    R_{m-1}
    \geq \bigg((2r_{m-1}-1)\frac{N_{m-1,0}}{N_{m-1}}+2\bigg) \bbclr{1+\frac{2r_{m-1}-1}{2}\frac{N_{m-1,0}}{N_{m-1}}}^{-1} - 2= 2-2=0,
\end{align*}
 If $r_{m-1}=1$ then $ \E[N_{m, 0}\mid G_{m-1}]\le  N_{m-1, 0} -  R_{m-1}\le  N_{m-1, 0}.$ Also, if $r_{m-1}<1$, then $p_{m-1}$ in \eqref{eq:pm} is chosen 
so that 
\begin{align*}
    0\leq (1-r_{m-1})q_{m-1} p_{m-1} \leq R_{m-1}.
\end{align*}
Therefore, $\E (N_{m,0}\mid G_{m-1}) \le N_{m-1,0}{-R_{m-1}+ (1-r_{m-1})p_{m-1}q_{m-1}}\le N_{m-1,0}$ if $r_m<1$, 
   showing that $N_{m,0}$ is a (non-negative) supermartingale. It follows that $N_{m,0}$ converges almost surely to some limit $\rho_0$ such that $\E\rho_0\le  \E N_{m_0,0}$.
    
    To show convergence in probability,  for any fixed $\eps>0$,
    \begin{align*}
        \IP\bigg( \frac{N_{m,0}}{N_m} >\eps\bigg)
        &\le  \IP\bigg(N_{m,0} >\frac{\eps m}{2} \bigg) +  \IP\bigg(N_m< \frac{m}{2} \bigg).\numberthis\label{eq:goodbadevents}
    \end{align*} 
Since $q_m\ge 2/3$ for all $m\ge m_0$, $N_{m}$ stochastically dominates a Bin$(m-m_0,2/3)$ random variable. 
By applying Markov's inequality and Chernoff's inequality for binomial distribution (see e.g.\ \cite{okamoto1959}) to the right-hand side of \eqref{eq:goodbadevents},  
     there is a positive constant $C$  such that
    \begin{align}
         \IP\bigg( \frac{N_{m,0}}{N_m} >\eps\bigg)< \frac{2\E N_{m_0,0}}{m\eps} + e^{-Cm} \rightarrow  0 \quad \text{as $m\to\infty$.} \label{eq:MC}
    \end{align}
    Thus, 
    $N_{m,0}/N_m\to 0$ in probability as $m\to\infty$. 

    The proof of the assertion under the assumptions in Item 2
    is similar.  By \eqref{r1} and letting 
     \begin{align*}
    {\widetilde R}_{m-1}&:= R_{m-1}+(1-q_{m-1})\frac{N_{m-1,0}}{N_{m-1}} =q_{m-1}\bbclr{\frac{N_{m-1,0}}{N_{m-1}}(2r_{m-1}-2)+2} + \frac{N_{m-1,0}}{N_{m-1}} -2,
    \end{align*}
     it follows that
   \begin{align}\label{r2}
    \E[N_{m, 0}\mid G_{m-1}]\le  N_{m-1, 0} - {\widetilde R}_{m-1} 
 +(1-r_{m-1})p_{m-1}q_{m-1}. 
   \end{align}
    If $r_{m-1}=1$ then $ \E[N_{m, 0}\mid G_{m-1}]\le  N_{m-1, 0} - {\widetilde R}_{m-1}.$ So for  $q_{m-1}$ as in \eqref{eq:qmrm},
    \begin{align*}
        {\widetilde R} _{m-1}
        &\geq \bbclr{2-\frac{N_{m-1,0}}{N_{m-1}}}\bbclr{\frac{N_{m-1,0}}{N_{m-1}}(2r_{m-1}-2)+2} \bbclr{2-2(1-r_{m-1})\frac{N_{m-1,0}}{N_{m-1}}}^{-1} + \frac{N_{m-1,0}}{N_{m-1}} -2 \\
        &= 2-\frac{N_{m-1,0}}{N_{m-1}}+ \frac{N_{m-1,0}}{N_{m-1}} -2 = 0,
    \end{align*}
   and thus $\E [ N_{m,0}\mid G_{m-1} ]\le N_{m-1,0}.$
If $r_{m-1}<1$, for $p_{m-1}$ as in \eqref{eq:pm1} we also have
    \begin{align*}
    {  (1-r_{m-1})q_{m-1}  p_{m-1} \bbclr{1-\frac{N_{m-1,0}}{N_{m-1}}}\le  (1-r_{m-1})q_{m-1}  p_{m-1} \leq } {\widetilde R}_{m-1}.
    \end{align*}
  Plugging the last two inequalities into \eqref{r2} yields $\E(N_{m,0}\mid G_{m-1})\leq N_{m-1,0}$, 
  and so $N_{m,0}$ is a non-negative supermartingale. To prove $N_{m,0}/N_m\to 0$ in probability, note that $N_m$ stochastically dominates a $\mathrm{Bin}(m-m_0, 1/2)$ variable in this case. An argument similar to that in  \eqref{eq:goodbadevents} and \eqref{eq:MC} then shows that $N_{m,0}/m\to 0$ in probability as $m\to\infty$. 
\end{proof}

\begin{proof}[Proof of Proposition \ref{prop:modelc2}]
As in the previous proofs, we start from the recursion formula \eqref{eq:new_nm0_diff3} for $N_{m,0}$ and use $N_{m-1,1}/N_{m-1} \le 1 $ and $ \sum_{j \ge 1}p_{m-1}^{j} N_{m-1,j}/N_{m-1} \le  p_{m-1}$ to obtain
\begin{multline*}
 \mathbb{E}( N_{m, 0} \mid G_{m-1}) 
    \le N_{m-1,0} \left( 1 + \frac{q_{m-1}(1-2r_{m-1}) - (1 - q_{m-1})}{N_{m-1}}\right) 
    \\ +p_{m-1} q_{m-1}(1-r_{m-1})+ 2 (1-q_{m-1}).
\end{multline*}
Note that $q_k(1-2r_k) - (1-q_k) = 2q_k(1-r_k) - 1 \le 0$ as long as $r_k \ge   1/2$ or $q_k \le (2(1-r_k))^{-1}$, which is assumed to hold for all $k\ge m_0$. Hence for all $m\ge m_0+1$,
\begin{align} \label{eq:supermg1}
 \mathbb{E}( N_{m, 0} \mid \cF_{m-1}) 
    &\le N_{m-1,0} 
 +p_{m-1} q_{m-1}(1-r_{m-1})+ 2 (1-q_{m-1}).
\end{align}
Iterating \eqref{eq:supermg1}, we obtain
    \begin{equation}
           \E(N_{m,0}) \le  N_{m_0,0} + \sum^{m-1}_{k=m_0}\{p_kq_k(1-r_k)+ 2(1-q_k)\}.
    \end{equation}
    So if $\sum_{k\ge   m_0}\{p_k(1-r_k)+ 2(1-q_k)\}<\infty$,  $\E(N_{m,0})$ is bounded in $m$. The claim that $N_{m,0}/m\to 0$ in probability follows immediately from Markov's inequality. 
    
    To show that $N_{m,0}/N_m\to 0$ in probability also, we write 
\begin{align}\label{eq:prod}
    \frac{N_{m,0}}{N_m} = \frac{N_{m,0}}{m} \frac{m}{N_m}.
\end{align}
As $N_m$ is distributed as $m_0$ plus a sum of independent Bernoulli variables, each with success probability $q_k$, it follows that 
$$\E N_m = m_0 + \sum_{k={m_0}}^{m{-1}} q_k;  \quad \mathrm{Var} (N_m) = \sum_{k={m_0}}^{m{-1}} q_k (1-q_k) \le \sum_{k={m_0}}^{m{-1}}  (1-q_k).$$ Under \eqref{eq:pqr}, $\sum_{k\ge m_0}(1-q_k)<\infty$, so there is a constant $C>0$ such that $\mathrm{Var} (N_m) \le C$ for all $m$, and $\lim_{m\to\infty}\E( N_m)/m=\lim_{m\to\infty}m^{-1}\sum^{m-1}_{k= m_0}(1-(1-q_k))=1$. For any $\eps>0$, 
\begin{align*}
    \IP\bclr{|N_m/m - 1 |\ge 2\eps} &\le   \IP\bclr{|N_m-\E N_m|/m + |\E N_m/m-1|\ge 2\eps}.
\end{align*}
Thus by choosing $m$ large enough such that $|\E N_m/m-1|\le \eps$ and using Markov's inequality,
\begin{align}
      \IP\bclr{|N_m/m - 1 |\ge 2\eps}\le \IP\bclr{|N_m-\E N_m|/m \ge \eps} \le \frac{\var (N_m)}{m^2\eps^2} \le \frac{C}{m^2\eps^2} \quad \text{as $m\to\infty$.}
\end{align}
On the other hand, $\E(N_{m,0})$ is bounded in $m$ and so $N_{m,0}/m \to 0$ in probability as $m \to \infty$. Thus, by Slutsky's Theorem and \eqref{eq:prod}, ${N_{m,0}}/{N_m} \to 0$ in probability as $m \to \infty$.
\end{proof}

\begin{proof}[Proof of Proposition \ref{prop:recursion}]
We start by proving \eqref{eq:new_nm0_diff}, the recursion for the number of isolated vertices. Below, we refer to the uniformly chosen vertex and the newly added vertex in each duplication-divergence step respectively as a \emph{parent} and a \emph{child}. We start by enumerating the (not necessarily mutually exclusive) events that lead to changes in the count of isolated vertices.
    \begin{enumerate}
        \item {Given that the parent is isolated}, 
        {the child and the parent either remain isolated,  with probability $1-r$; or they form an (isolated) edge,  with probability $r$.}
        \item Given that the parent has degree $j \geq 1$, the child does not connect to $v$ and nor to any of the $j$ neighbours of $v$; this happens with probability $(1-r)p^{j}$. 
        \item A non-isolated vertex is chosen for a deletion-addition step and hence becomes isolated.  
        \item A vertex of degree one has its only neighbour chosen for a deletion-addition step, causing both the neighbour and the vertex to become isolated. Given $G_{m-1}$, this occurs with conditional probability $(1-q)N_{m-1,1}/N_{m-1}$.   
    \end{enumerate}
Furthermore, in Model A, a new isolated vertex is added to the existing network in a deletion-addition step, regardless of the degree of the uniformly chosen vertex.  
Combining these events, 
we obtain
\begin{align*}
    \E[N_{m,0}\mid G_{m-1}] &= N_{m-1,0} - qr\frac{N_{m-1,0}}{N_{m-1}}+ q(1-r)\frac{N_{m-1,0}}{N_{m-1}} + q(1-r)\sum^{m-2}_{j =    1}\frac{N_{m-1,j}}{N_{m-1}}p^j\\
    &\quad + (1-q) \sum^{m-2}_{j = 1}\frac{N_{m-1,j}}{N_{m-1}} + (1-q)\frac{N_{m-1,1}}{N_{m-1}} + (1-q);
\end{align*}
noting that $\sum_{j \ge   1}N_{m-1,j}=N_{m-1}-N_{m-1,0}$ and rearranging the terms proves \eqref{eq:new_nm0_diff}. 

For $N_{m,k}$, $k \ge   1$, we start by listing the (not necessarily mutually exclusive) events that add to the count of vertices of degree $k$.
\begin{enumerate}
    \item The parent has degree $k-1$ just before undergoing the duplication-divergence step. The parent will have
    degree $k$ if it forms an edge with the new vertex. 
    Additionally, the child will have degree $k$ if it connects to its parent and all neighbours of its parent. 
Given $G_{m-1}$, the expected number of vertices of degree $k$ gained in this way is 
    \begin{equation}\label{eq:dup1}
        q\big(r+r(1-p)^{k-1}\big)\frac{N_{m-1,k-1}}{N_{m-1}}.
    \end{equation}
    
    \item A neighbour $v$ to some vertex $u$ of degree $k-1$ is selected as a parent, and $u$ gains a new edge if the child of $v$ connects to $u$. Given $G_{m-1}$, the expected number of vertices of degree $k-1$ that gain an edge in this way is
    \begin{equation}\label{eq:secondary}
    \frac{q(1-p)(k-1)N_{m-1,k-1}}{N_{m-1}}.
    \end{equation}
    
    \item Given that the parent has degree $j \geq {k}$, the number of its neighbours connecting to the child has the $\mathrm{Bin}(j, 1-p)$ distribution. In this case, the probability that the child has degree $k$ after the duplication-divergence step, given $G_{m-1}$, is thus
    \begin{multline}\label{eq:deggk}
    qr\sum^{m-1}_{j= k} \frac{N_{m-1,j}}{N_{m-1}} \binom{j}{k-1} p^{j-k+1} (1-p)^{k-1}
   + q(1-r)\sum^{m-1}_{j= k} \frac{N_{m-1,j}}{N_{m-1}} \binom{j}{k} p^{j-k} (1-p)^{k}.
   \end{multline}

\item If a neighbour $v$ of some vertex $u$ of degree $k+1$ is uniformly selected for a deletion-addition step, vertex $u$ has degree $k$ after the edge between $u$ and $v$ is deleted. Given $G_{m-1}$, the expected number of vertices of degree $k+1$ that lose an edge in this way is 
\begin{equation}
\label{eq:increaseDel}
    (1-q)(k+1)\frac{N_{m-1,k+1}}{N_{m-1}}.
\end{equation}
\end{enumerate}
The count of vertices of degree $k$ can decrease too, if one or more of the following events happen.
\begin{enumerate}
    \item The parent has degree $k$ and forms an edge with the new vertex. {Given $G_{m-1}$, the probability is}
    \begin{equation}\label{eq:del1}
        \frac{qrN_{m-1,k}}{N_{m-1}}.
    \end{equation}

    \item {At a duplication-divergence step, the parent is a neighbour of some vertex $u$ of degree $k$, and the child forms an edge with $u$, adding one to the degree of $u$. Given $G_{m-1}$, the expected number of vertices of degree $k$ that gain an edge in this way is}
    \begin{equation}\label{eq:del2}
       q(1-p)\frac{kN_{m-1,k}}{N_{m-1}}.
    \end{equation}

    \item A vertex of degree $k$ or its neighbour is selected for the deletion-addition step. Given $G_{m-1}$, the expected number of vertices of degree $k$ lost in this way is
    \begin{equation}\label{eq:del3}
        (1-q) \left(\frac{N_{m-1,k}}{N_{m-1}} + \frac{kN_{m-1,k}}{N_{m-1}} \right).
    \end{equation}
\end{enumerate}

Assembling the terms \eqref{eq:dup1}--\eqref{eq:del3}, and noting the contributions from \eqref{eq:del1}, \eqref{eq:del2}, and \eqref{eq:del3} are negative then yield \eqref{eq:new_nmk_diff}.

The proof of the corresponding statements for Model B is a straightforward adaptation and hence omitted.
\end{proof}

\begin{proof}[Proof of Proposition \ref{prop:isolatedbound}]
Let $\widetilde \cN^{(k)}_{m,j}$ be the set of vertices in Model A with \emph{exactly} $j$ neighbours of degree $k$ at time $m$, where $j,k \ge  1$,
and let $\widetilde N^{(k)}_{m,j}:=|\widetilde \cN^{(k)}_{m,j}|$. The sets $\widetilde \cN^{(k)}_{m,j}$ and $\widetilde \cN^{(k)}_{m,l}$ are disjoint for $j \ne l,$ and so 
$\cup_{j=0}^{m-1}  \widetilde \cN^{(k)}_{m,j} = [m]$. To prove \eqref{eq:isolatedbd}, we first condition on the duplication-divergence step and the deletion-and-addition step to get
\begin{align}
    \E[(N_{m,0}-N_{m-1,0})^2\mid G_{m-1}]
    &= q 
    \E[(N_{m,0}-N_{m-1,0})^2\mid G_{m-1}; \mbox{duplication-and-divergence}] \notag \\
    &+ (1-q)  \E[(N_{m,k}-N_{m-1,k})^2\mid G_{m-1}; \mbox{deletion-and-addition}]. \label{eq:ddsplit}
\end{align}
To bound the conditional expectations on the right-hand side of \eqref{eq:ddsplit}, we observe that, given $G_{m-1}$, the change in the count of isolated vertices is
\begin{enumerate}
\item $-1$, $0$, or $+1$ during a duplication-divergence step. If the parent is an isolated vertex, the change is $-1$ if a parent-child link is formed; otherwise, it is $+1$. If the parent is a non-isolated vertex and the child fails to connect to its parent, the change is $+1$, otherwise it is 0; 
\item \label{item:sq2} $+(k+2)$, with $k \geq 0$, when a vertex in $\mathcal{\widetilde{N}}^{(1)}_{m,k}\setminus \mathcal{{N}}_{m-1,0}$ is chosen for a deletion-addition step, as during a deletion-addition step,  a new isolated vertex is added and the chosen vertex becomes isolated too. The conditional probability of this event, given $G_{m-1}$, is at most $(1-q)\widetilde{N}^{(1)}_{m-1,k}/N_{m-1}$; 
\item $+1$ when an isolated vertex is selected for a deletion-addition step.
\end{enumerate}

The first item implies that $\E[(N_{m,0}-N_{m-1,0})^2\mid G_{m-1}; \mbox{duplication}]\le 1$. Combining this and the bounds implied by the other items, \eqref{eq:ddsplit} can be bounded as
\begin{align*}
    \E[(N_{m,0}-N_{m-1,0})^2\mid G_{m-1}]
     \le q +(1-q)\sum^{m-2}_{k = 0}{(k+2)}^2\frac{\widetilde N^{(1)}_{{m-1},k}}{m-1} + (1-q)\frac{N_{m,0}}{{m-1}};  
\end{align*}
applying 
${N_{m,j}} \le 
m$,  $\widetilde N^{(1)}_{m,k}=\sum^{m-1}_{j =    k} |\widetilde \cN^{(1)}_{m,k}\cap \cN_{m,j}|$ and $(x+1)^2\leq 5x^2+4$ for non-negative integers $x$, we find
\begin{align*}
      \E[(N_{m,0}-N_{m-1,0})^2\mid G_{m-1}] &\le (1-q) \sum^{m-2}_{k =   0 }\sum^{m-2}_{j =   0} (k+2)^2\frac{|\widetilde \cN^{(1)}_{{m-1},k}\cap \cN_{{m-1},j}|}{{m-1}} \mathbf{1}[j\ge   k]  + 1
    \nonumber \\
    &\le   (1-q)\sum^{m-2}_{j =   0} \sum^{m-2}_{k =   0} (j+2)^2\frac{|\widetilde \cN^{(1)}_{{m-1},k}\cap \cN_{{m-1},j}|}{{m-1}}  + 1 \\
     &\le  (1-q)\sum^{m-2}_{j =   0}(j+2)^2\frac{ N_{{m-1},j}}{{m-1}} + 1\\
 &\le 5(1-q)  \sum^{m-2}_{j= 0}  j^2 \frac{ N_{{m-1},j}}{{m-1}} + 5,
\end{align*}
which is  \eqref{eq:isolatedbd}. For $k\ge 2$, we again consider the cases of duplication-and-divergence and deletion-and-addition separately. Denote by $V$ the parent chosen at time $m$, and recall that the parent-child link is formed independently of the other edges of the child. Hence,
\begin{eqnarray*}
\lefteqn{q 
    \E[(N_{m,k}-N_{m-1,k})^2\mid G_{m-1}; \mbox{duplication-divergence}]} \\
    &=&
    q \sum_{l=0}^{m-2} \sum_{j=0}^{m-2} \sum_{s=0}^{m-2} 
    \PP \big(V \in \cN_{m-1, l} \cap \widetilde{\cN}_{m-1,j}^{(k)} \cap \widetilde{\cN}_{m-1,s}^{(k-1)}\mid  G_{m-1} \big) \mathbf{1}[j + s \le l] \\
    && 
\quad \cdot  \E\big[(N_{m,k}-N_{m-1,k})^2\mid G_{m-1}; \mbox{duplication-divergence};V \in \cN_{m-1, l}\cap \widetilde{\cN}_{m-1,j}^{(k)}\cap \widetilde{\cN}_{m-1,s}^{(k-1)}\big] \\
 &=&
    q \sum_{l=0}^{m-2} \sum_{j=0}^{m-2} 
    \sum_{s=0}^{m-2} 
    \frac{\big|  \cN_{m-1, l}\cap \widetilde{\cN}_{m-1,j}^{(k)}\cap \widetilde{\cN}_{m-1,s}^{(k-1)}\big|}{{m-1}}\mathbf{1}[j + s \le l]
\sum_{a=0}^j \sum_{b=0}^s \sum_{c=0}^{l-j-s} p(a,b,c)   \\
&& \quad \cdot\big\{r\big[ \mathbf{1}( l=k-1) - \mathbf{1}( l=k)+ \mathbf{1}(a+b+c+1=k)  -a + b \big]^2+   (1-r)\big[ \mathbf{1}(a+b+c=k)  -a + b \big]^2\big\}.
\end{eqnarray*}
 where $p(a,b,c)$ is the probability of the event that the child retains $a$ edges to neighbours of $V$ with degree $k$, $b$ edges to those with degree $k-1$, and $c$ edges to other neighbours of $V$. 
 For $0 \le a,b \le  l$, 
 \begin{align*}
     \big( \mathbf{1}( l=k-1) - \mathbf{1}( l=k)+ \mathbf{1}(a+b+c+1=k)  -a + b \big)^2 
 &\le  (b-a)^2 + 4|b-a|  + 4 \\
 &\le l^2 + 4l + 4 \le 5l^2 + 4;
 \end{align*}
 where the last inequality is valid for non-negative integers $l$. Similarly, $\big(\mathbf{1}(a+b+c=k)  -a + b \big)^2 \le 3l^2 +1.$ Thus, noting that $\sum^j_{a=0}\sum^s_{b=0}\sum^{l-j-s}_{c=0}p(a,b,c)=1$, 
 \begin{eqnarray}\label{dd}
    \lefteqn{q\E[(N_{m,k}-N_{m-1,k})^2|G_{m-1}; \mbox{duplication}]}\notag\\ 
    &\le & q \sum_{l=0}^{m-2} (5l^2+4)\sum_{j=0}^{m-2} 
    \sum_{s=0}^{m-2} 
    \frac{\big|  \cN_{m-1, l}\cap \widetilde{\cN}_{m-1,j}^{(k)}\cap \widetilde{\cN}_{m-1,s}^{(k-1)}\big|}{{m-1}}\mathbf{1}[j + s \le l]\cdot \notag\\
    &=& 5q\sum_{l=0}^{m-2} l^2
 \frac{  N_{m-1, l}}{{m-1}} + 4q . 
\end{eqnarray}

On the other hand, the overall change in the count of vertex of degree $k$ is $s-j-\mathbf{1}[l=k]$, if a vertex in $\cN_{m-1,l}\cap \mathcal{\widetilde N}^{(k)}_{m-1,j}\cap \mathcal{\widetilde N}^{(k+1)}_{m-1,s} $ is selected for deletion-addition. So by a similar argument, we obtain
\begin{eqnarray}\label{da}
 \lefteqn{  (1-q)  
    \E[(N_{m,k}-N_{m-1,k})^2\mid G_{m-1}; \mbox{deletion}]}\notag\\
    &=& 
  (1-q)  \sum_{l=0}^{m-1} \sum_{j=0}^{m} 
    \sum_{s=0}^m 
    \frac{\big|  \cN_{m-1, l}\cap \widetilde{\cN}_{m-1,j}^{(k)}\cap \widetilde{\cN}_{m-1,s}^{(k+1)}\big|}{{m-1}}
  \mathbf{1}[j + s \le l]\cdot (s-j-\mathbf{1}[l=k])^2 \notag \\
&\le &  (1-q)  \sum_{l=0}^{m-2} 
(3l^2+1)\sum_{j=0}^{m-1} \sum_{s=0}^{m-1} 
    \frac{\big|  \cN_{m-1, l}\cap \widetilde{\cN}_{m-1,j}^{(k)}\cap \widetilde{\cN}_{m-1,s}^{(k+1)}\big|}{{m-1}}
  \mathbf{1}[j + s \le l] \notag \\ &\le &  3(1-q)  \sum_{l=0}^{m-2} l^2
    \frac{ N_{m-1, l}}{{m-1}} + 1-q.  
\end{eqnarray}
By \eqref{dd} and \eqref{da}, 
\begin{eqnarray*}
      \E[(N_{m,k}-N_{m-1,k})^2\mid G_{m-1}
      ]
 &\le &  (3+2q) \sum_{l=0}^{m-2} l^2 
 \frac{  N_{m-1, l}}{{m-1}} + 4,
\end{eqnarray*}
proving \eqref{eq:sqdiffk}. 
\end{proof}

\begin{proof}[Proof of Proposition \ref{lemmafirstmoment}]
We start by proving the claim for $\psi_1(m)$. Clearly,
    \begin{align}\label{eq:mom1ineq}
    m^{-1}\E [N_{m,k}] \le (m-1)^{-1}\E[\E [N_{m,k}\mid G_{m-1}]].
\end{align}
By the recursion formula \eqref{eq:new_nmk_diff} for $N_{m,k}$ with $k \ge   1$,
\begin{align}
\label{eq:new_nmk_diff3}
   &\frac{\mathbb{E}\left[{N_{m, k} }  \mid G_{m-1}\right]}{{m-1}} = \frac{N_{m-1, k} }{{m-1}}  +  \frac{N_{m-1,k-1}}{(m-1)^2}q\left(r+r(1-p)^{k-1}+(1-p)(k-1)\right) \nonumber \\
   &\quad -\frac{N_{m-1,k}}{(m-1)^2}q\big(r+(1-p)k\big)+q(1-r) \sum^{m-2}_{j =   k}\frac{N_{m-1,j}}{(m-1)^2} {j \choose k}p^{j-k}(1-p)^k\nonumber\\
    &\quad +q r\sum^{m-2}_{j =   k}\frac{N_{m-1,j}}{(m-1)^2} {{j} \choose {k-1}}p^{j-k+1}(1-p)^{k-1}  + (1-q) (k+1)\frac{N_{m-1,k+1}-N_{m-1,k}}{(m-1)^2}.
\end{align} 
Recall $\psi_i(m)$ in \eqref{de:psi}. By  
multiplying both sides of \eqref{eq:mom1ineq} by $k$, applying  \eqref{eq:new_nmk_diff3}, summing over $1\le k\le m-1$, and a straightforward simplification,  
\begin{align*}
     \psi_1(m) &\le  \psi_1(m-1)\\
    & \quad  + q\sum_{k=1}^{{m-1}}k\bbbclr{\frac{\E N_{m-1,k-1}}{ (m-1)^2}\big(r+{r(1-p)^{k-1}} + (1-p)(k-1)\big)-\frac{\E N_{m-1,k}}{(m-1)^2}}(r+(1-p)k)\\
    &\quad +q(1-r)\sum_{k=1}^{m-1}k\sum^{m-2}_{j= k}\frac{\E N_{m-1,j}}{(m-1)^2}  {j \choose k}p^{j-k}(1-p)^k  \\
    &\quad +(1-q)\sum_{k=1}^{m-1} k(k+1)\E\bbcls{\frac{N_{m-1,k+1}-N_{m-1,k}}{(m-1)^2}} \\
    &\quad +qr\sum_{k=1}^{m-1}k\sum^{m-2}_{j= k}\frac{\E N_{m-1,j}}{(m-1)^2} {{j} \choose {k-1}}p^{j-k+1}(1-p)^{k-1}\\
    &= \psi_1(m-1) + q r \sum^{m-2}_{k=0} \frac{\E N_{m-1,k}}{(m-1)^2} + \sum^{m-2}_{k=0} \frac{\E N_{m-1,k}}{(m-1)^2} k (q(1-p) -2(1-q)) \\
    &\quad + q(1-r)\sum^{m-2}_{j=1}\frac{\E N_{m-1,j}}{(m-1)^2}\sum^j_{k=1} {j \choose k}p^{j-k}(1-p)^k k\\
    &\quad + qr\sum^{m-2}_{j=0}\frac{\E N_{m-1,j}}{(m-1)^2}\sum^j_{k=0} {j \choose k}p^{j-k}(1-p)^k (k+1),    \numberthis\label{psi1}
\end{align*}
Noting that $ \sum^j_{k=0} {j \choose k}p^{j-k}(1-p)^k k = j (1-p)$, another routine calculation yields
\begin{align*}
   \psi_1(m) &\le \psi_1(m-1) +
   2 [ q(1-p) - (1-q)]
   \sum^{m-2}_{k=0} k \frac{\E N_{m-1,k}}{(m-1)^2} +  2qr \sum^{m-2}_{k=0} \frac{\E N_{m-1,k}}{(m-1)^2}\\
     &= \psi_1(m-1) +\bclr{4q-2pq-2} \sum^{m-2}_{k=0} k \frac{\E N_{m-1,k}}{(m-1)^2} + \frac{2 qr}{m-1} \\
   & = \psi_1(m-1) \left( 1 +\frac{\kappa}{m-1}  \right) +  \frac{2 qr}{m-1}, \numberthis\label{eq:genpsi2}
\end{align*}
where $\kappa:=4q-2pq-2$, as in \eqref{kappa}. 
Iterating \eqref{eq:genpsi2} gives, for $m\ge m_0+1$, 
\begin{eqnarray*}
   \psi_1(m)     &\le& 
  \sum_{j=m_0}^{m-1} \frac{2qr}{j} \prod_{\ell=j+1}^{m-1} \left( 1 + \frac{\kappa}{\ell}\right) + \psi_1(m_0) \prod_{j=m_0}^{m-1} \left( 1 + \frac{\kappa}{j} \right) \\
  &=& \sum_{j=m_0}^{m-1} \frac{2qr}{j} \frac{\Gamma( m + \kappa)}{\Gamma ( j + \kappa +1)}
\frac{\Gamma(j + 1 )}{\Gamma ( m)}
  + \psi_1(m_0)  \frac{\Gamma( m + \kappa)}{\Gamma ( m_0+\kappa )}
\frac{\Gamma( {m_0} )}{\Gamma ( m)}.
\end{eqnarray*}
 Since $\Gamma(m+\kappa)/\Gamma(m)=m^\kappa(1+O(m^{-1}))$ by Stirling's formula, it follows that 
\begin{align*}
\psi_1(m)\le Cm^\kappa \sum_{j=m_0}^{m-1} j^{-(1+\kappa) } + Cm^\kappa\le 
\begin{cases}
    C&\text{if $\kappa<0$;}\\
    C\log m&\text{if $\kappa=0$;}\\
    Cm^\kappa &\text{if $\kappa>0$.}
\end{cases}
\end{align*}

For $\psi_2(m)$ the argument is very similar. Repeating the steps for obtaining \eqref{psi1}, we get
\begin{align*}
     \psi_2(m) &\le  \psi_2(m-1)  + \sum_{k=0}^{m-2}\frac{\E N_{m-1,k}}{(m-1)^2}\bclc{q(k+1)^2(r+(1-p)k)-qk^2(r+(1-p)k)}\\
     &\quad + \sum_{k=2}^{m-2}\frac{\E N_{m-1,k}}{(m-1)^2}(1-q)((k-1)^2k-k^2(k+1))\\
     & \quad +q (1-r)\sum^{m-2}_{j=0}\frac{\E N_{m-1,j}}{(m-1)^2}\sum^j_{k=1}\binom{j}{k} p^{j-k}(1-p)^kk^2 \\
    &\quad + q r\sum^{m-2}_{j=0}\frac{\E N_{m-1,j}}{(m-1)^2}\sum^j_{k=0}\binom{j}{k} p^{j-k}(1-p)^k(k+1)^2 . 
\end{align*}
For $Y_j\sim \mathrm{Bin}(j,1-p)$, 
$\E[(Y_j+1)^2]=j^2(1-p)^2+jp(1-p)+2j(1-p)+1,$ and 
$ \E [Y_j^2] = j p(1-p) + j^2 (1-p)^2$. Thus, by a similar computation as for $\psi_1(m)$ and with $\tau$ as in \eqref{tau}, the last display can be simplified as
\begin{align*}
     \psi_2(m) &\le  \psi_2(m-1)  + \sum_{k=0}^{m-2}\frac{\E N_{m-1,k}}{(m-1)^2}\bclc{q(k+1)^2(r+(1-p)k)-qk^2(r+(1-p)k)}\\
     &\quad + \sum_{k=2}^{m-2}\frac{\E N_{m-1,k}}{(m-1)^2}(1-q)((k-1)^2k-k^2(k+1))\\
     & \quad +q (1-r)\sum^{m-2}_{k=1}\frac{\E N_{m-1,k}}{(m-1)^2} (k p(1-p) + k^2 (1-p)^2) 
    \\
    &\quad + q r\sum^{m-2}_{k=0}\frac{\E N_{m-1,k}}{(m-1)^2} (k^2(1-p)^2+kp(1-p)+2k(1-p)+1) \\
    &= \psi_2(m-1) \bbclr{1+\frac{\tau}{m-1}}  
     + \frac{\lambda}{m-1} \psi_1 (m-1) +\frac{2qr}{m-1},
\end{align*}
where we also set $\lambda =2qr+q(1-p)+1-q+qp(1-p)+2qr(1-p)$. Iterating over $m$ gives 
\begin{align}
   \psi_2(m) &\le  \psi_2(m_0) \prod_{j=m_0}^{m-1}  \left(1  + \frac{\tau }{j} \right) + \sum_{j=m_0}^{m-1} \frac{\lambda}{j} \psi_1(j) \prod_{k=j+1}^{m-1} \left(1  + \frac{\tau }{k} \right) 
    + \sum_{j=m_0}^{m-1}\frac{2qr}{j} \prod_{k=j+1}^{m-1} \left(1  + \frac{\tau }{k} \right)\notag\\
    &= \psi_2(m_0) \frac{\Gamma(m+\tau)\Gamma(m_0)}{\Gamma(m_0+\tau)\Gamma(m)} +  \sum_{j=m_0}^{m-1} \frac{\lambda}{j}  \frac{\Gamma(m+\tau)\Gamma(j)}{\Gamma(j+\tau)\Gamma(m)}\psi_1(j) +  \sum_{j=m_0}^{m-1} \frac{2qr}{j}  \frac{\Gamma(m+\tau)\Gamma(j)}{\Gamma(j+\tau)\Gamma(m)}.\label{eq:psi2check}
\end{align}
This time, we apply the different bounds for $\psi_1$ depending on the value of $\kappa$. When $\kappa<0$, $\psi_1(j)\leq C$, and so a calculation similar to those for obtaining \eqref{eq:psi1bd} yields the bounds in \eqref{b1}.
When $\kappa=0$, we use the bound $\psi_1(m)\leq C\log m$ for the second term in \eqref{eq:psi2check}. This gives  
\begin{align*}
    \sum_{j=m_0}^{m-1} \frac{\lambda}{j}  \frac{\Gamma(m+\tau)\Gamma(j)}{\Gamma(j+\tau)\Gamma(m)}\psi_1(j) \leq \sum_{j=m_0}^{m-1} \frac{C}{j}  \frac{\Gamma(m+\tau)\Gamma(j)}{\Gamma(j+\tau)\Gamma(m)} \log j\leq 
     \begin{cases}
         C\log m&\text{if $\tau<0$;}\\
         C\log^2 m&\text{if $\tau=0$;}\\
         C    m
         ^{\tau}
         &\text{if $\tau>0$,}
     \end{cases}
\end{align*}
which are the bounds in \eqref{b2}, and they dominate the  two other terms in \eqref{eq:psi2check}, regardless of the value of~$\tau$. Finally, if $\kappa>0$, $\psi_1(m)\leq Cm^\kappa$. Thus
\begin{align*}
      \sum_{j=m_0}^{m-1} \frac{\lambda}{j}  \frac{\Gamma(m+\tau)\Gamma(j)}{\Gamma(j+\tau)\Gamma(m)}\psi_1(j) \leq \sum_{j=m_0}^{m-1} \frac{C}{j}  \frac{\Gamma(m+\tau)\Gamma(j)}{\Gamma(j+\tau)\Gamma(m)} \log j\leq 
     \begin{cases}
          Cm^\kappa, &\text{if $\tau\leq 0$, $\kappa>\tau$;}\\
          Cm^\kappa &\text{if $\tau>0$ and $\kappa>\tau$;}\\
          Cm^\tau &\text{if $\tau>0$ and $\kappa<\tau$;}\\
          Cm^\tau \log m &\text{if $\tau=\kappa$.}
     \end{cases}
\end{align*}
Since the remaining terms in \eqref{eq:psi2check} can be bounded using the same upper bounds, this proves \eqref{b3}. 
\end{proof}

\section*{Funding}
GR is supported in part by EPSRC grant EP/T018445/1. TYYL is supported by the Sverker Lerheden Foundation, Knut and Alice Wallenberg Foundation, Ragnar  S\"oderberg Foundation and Swedish Research Council.

\section*{Appendix: Further Simulations}\label{Ssupp}

Here we present the simulations of Model A when $q=0.7,0.5,0.2$. The figures below show the evolution of the average proportion of isolated vertices of Model A over time, initialised either with a single edge or a combination of a single vertex and an edge. The simulations are run for 1000 steps, with 30 repeats. Different values of $p$ are represented by different colours; see each figure for the details. The \textbf{solid} lines represent parameter sets satisfying the conditions of Theorem \ref{th:LBk}, while the \textbf{dashed} lines correspond to cases where these conditions are not met. In particular, for $r = 0$, the theorem requires $q \leq 1/2$.  The shaded regions correspond to $\pm 2$ standard errors; capped at $0$ and $1$. The coloured bars on the right indicate the intervals $[\rho_0,(\rho_1\wedge 1)]$, where $\rho_0$ and $\rho_1$ are as in \eqref{de:rho} and \eqref{rho1}; \textbf{noting that $\rho_1$ does not depend on $p$ if $r=1$.}
Thus when $r=1$, the interval $[\rho_0,\rho_1]$ is represented by a single bar in magenta in the plots below, when applicable.

\begin{figure}[!h]
\centering 
   \includegraphics[width=0.85\textwidth, height=0.7\textwidth]{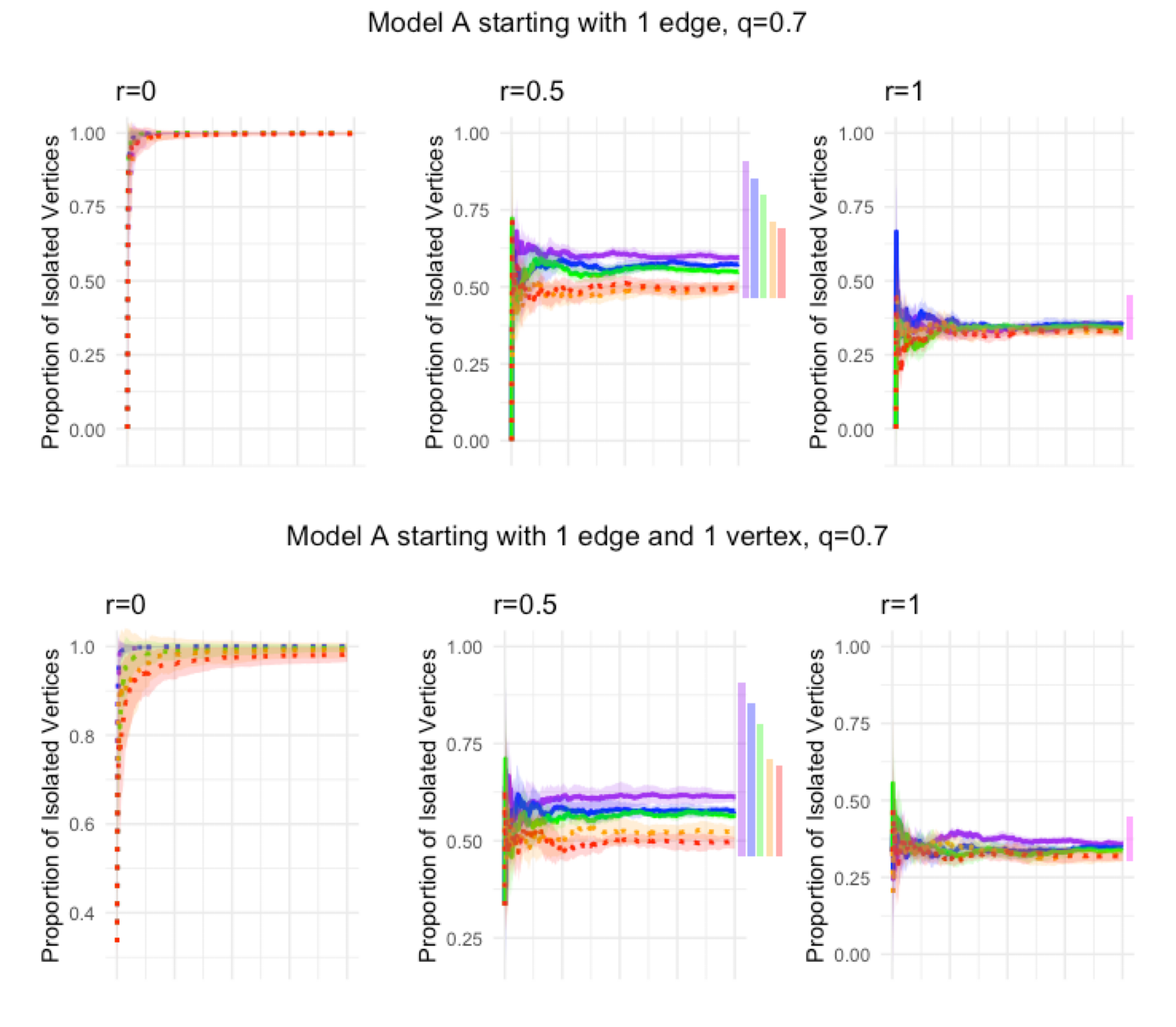} 
  \caption{\footnotesize 
  \textbf{Red}: $p=0$; \textbf{Orange}: $p=2 - \sqrt{182} / 7 \approx 0.0728$; \textbf{Green}: $0.4$; \textbf{Blue}: $p=0.6$; \textbf{Purple}: $p=0.8$. 
 For $r \neq 0$ and $q = 0.7$, Theorem 3.2 requires $p > 2 - \sqrt{182} / 7$, subject to $q \leq \min \{1, 1/(2(1-r))\}$. While the conditions for Theorem 3.2 are not met when $q=0.7$ $r=0$, the proportion of isolated vertices appear to converge to 1, even when $p$ is small. For $r=0.5,1$, the result of Theorem 3.2 seems to hold even when $p\le \sqrt{182}/7$ (in which case $\tau\ge 1$.) 
 }
  \label{fig:modelAq7}
  \vspace{-4mm}
\end{figure}

\begin{figure}[!h]
\centering 
   \includegraphics[width=0.82\textwidth, height=0.6\textwidth]{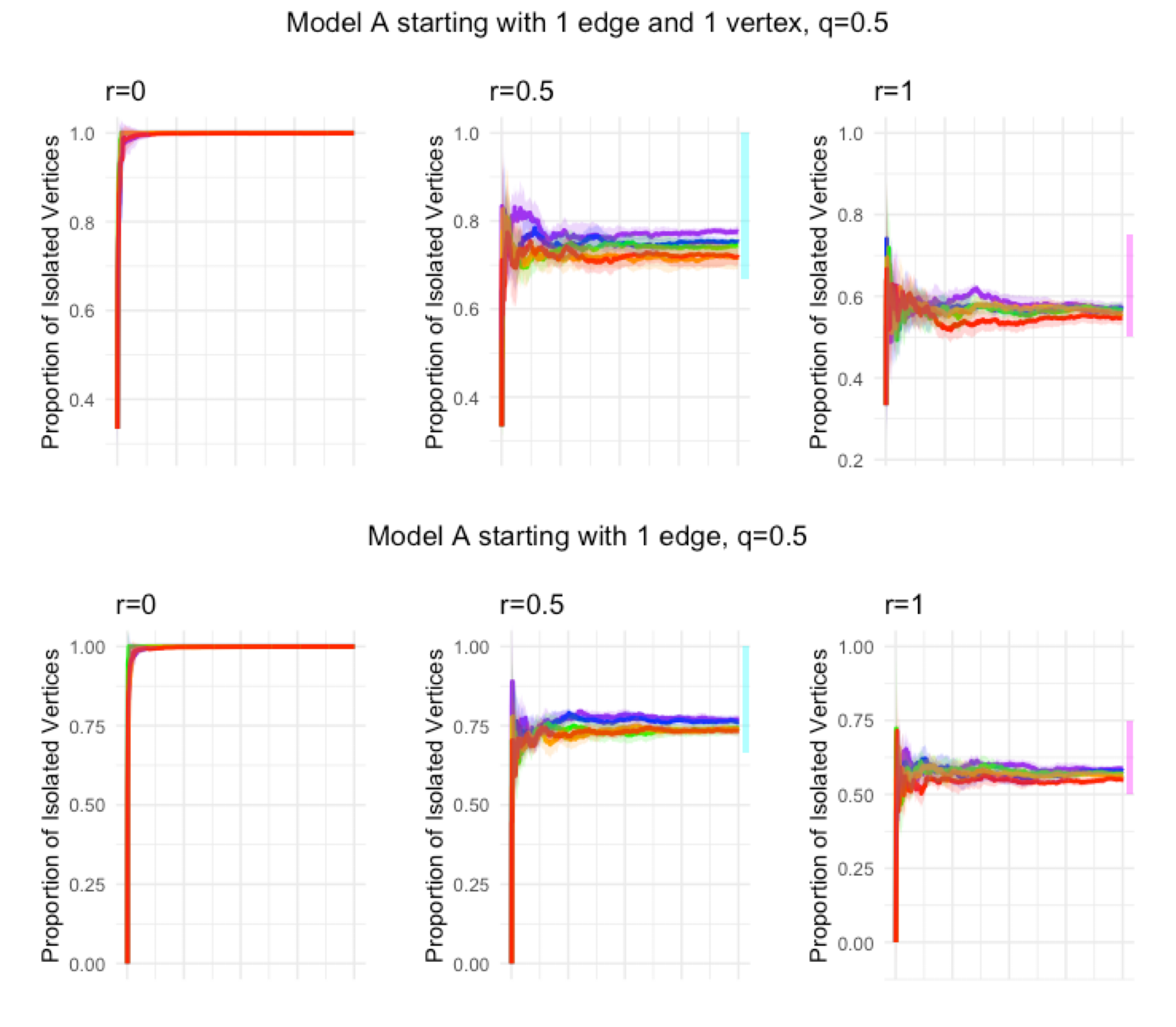} 
  \caption{\footnotesize  
  \textbf{Red}: $p=0$; \textbf{Orange}: $p=0.2$; \textbf{Green}: $0.4$; \textbf{Blue}: $p=0.6$; \textbf{Purple}: $p=0.8$. When $q=r=0.5$, $\rho_1\ge 1$ for all values of $p$ so we represent $[\rho_0,(\rho_1\wedge 1)]$ in this case by a single cyan bar. When $r=0.5,1$, the average proportion of isolated vertices appears to be closer to $\rho_0$ than to $\rho_1$, for all $p$ considered here. }
  \label{fig:modelAq5}
\end{figure}

\begin{figure}[!h]
\centering 
   \includegraphics[width=0.82\textwidth, height=0.6\textwidth]{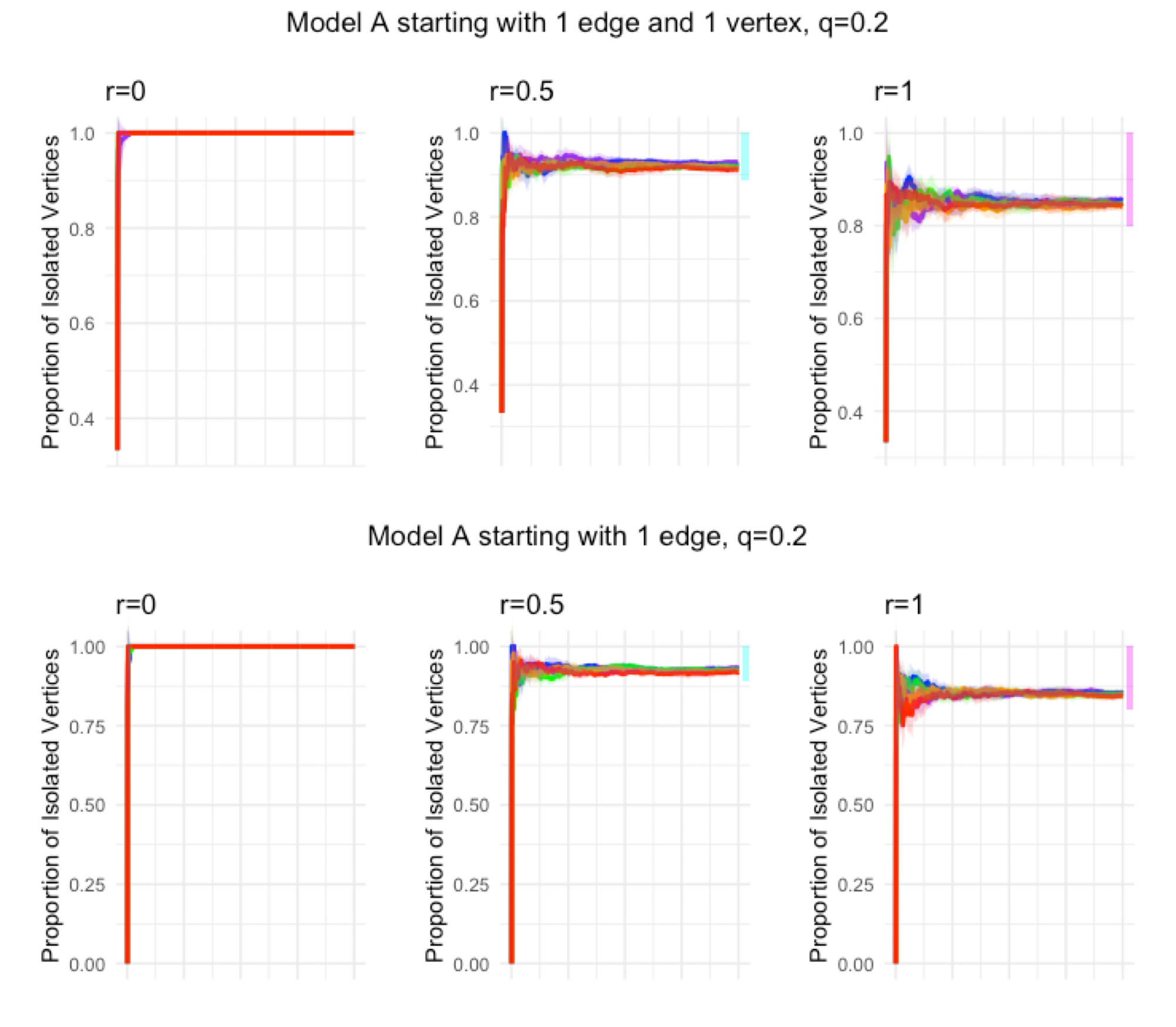} 
  \caption{\footnotesize  
  \textbf{Red}: $p=0$; \textbf{Orange}: $p=0.2$; \textbf{Green}: $0.4$; \textbf{Blue}: $p=0.6$; \textbf{Purple}: $p=0.8$; When $q=0.2$ and $r=0.5$, $\rho_1\ge 4/3$ for all values of $p$ so we represent $[\rho_0,(\rho_1\wedge 1)]$ in this case by a single cyan bar. When $q=0.2$ and $r=1$, $\rho_1=1.2\ge 1.$ When $r=0.5,1$, the average proportion of isolated vertices appears to be closer to $\rho_0$ than to $\rho_1$, for all $p$ considered here.}
  \label{fig:modelAq2}
\end{figure}

\clearpage


\bibliographystyle{abbrv} 
\bibliography{reference}       

\end{document}